\documentclass[12pt]{amsart}
\pdfoutput=1
\usepackage{amsmath, amssymb, amsthm, amsfonts, mathrsfs}
\usepackage{cite}
\usepackage{amscd}
\usepackage{url}
\usepackage{graphicx}
\usepackage{caption}
\usepackage{subcaption}
\usepackage{fullpage}

\usepackage{lscape, pdflscape}

\newtheorem{proposition}{Proposition}

\newtheorem{lemma}{Lemma}

\newtheorem{theorem}{Theorem}
\newtheorem{corollary}[proposition]{Corollary}
\theoremstyle{definition}

\newcommand{\ds}[1]{\ensuremath{\displaystyle{#1}}}

\newcommand{\incentive}{\varphi}

\makeatletter
\g@addto@macro{\endabstract}{\@setabstract}
\newcommand{\authorfootnotes}{\renewcommand\thefootnote{\@fnsymbol\c@footnote}}%
\makeatother
\begin{document}
\date{}

\begin{center}
  \LARGE 
  Stationary Stability for Evolutionary Dynamics in Finite Populations \par \bigskip

  \normalsize
  \authorfootnotes
  Marc Harper\footnote{corresponding author, email: marc.harper@gmail.com\\ \indent AMS keywords: 91A22, 92D25, 37B25}\textsuperscript{1}, Dashiell Fryer\textsuperscript{2}\par \bigskip

  \textsuperscript{1}Department of Genomics and Proteomics, UCLA, \par
  \textsuperscript{2}Department of Mathematics, Pomona College\par \bigskip
  
\end{center}


\begin{abstract}
We demonstrate a vast expansion of the theory of evolutionary stability to finite populations with mutation, connecting the theory of the stationary distribution of the Moran process with the Lyapunov theory of evolutionary stability. We define the notion of stationary stability for the Moran process with mutation and generalizations, as well as a generalized notion of evolutionary stability that includes mutation called an incentive stable state (ISS) candidate. For sufficiently large populations, extrema of the stationary distribution are ISS candidates and we give a family of Lyapunov quantities that are locally minimized at the stationary extrema and at ISS candidates. In various examples, including for the Moran and Wright-Fisher processes, we show that the local maxima of the stationary distribution capture the traditionally-defined evolutionarily stable states. The classical stability theory of the replicator dynamic is recovered in the large population limit. Finally we include descriptions of possible extensions to populations of variable size and populations evolving on graphs.
\end{abstract}

\newpage

\section*{Graphical Abstract}

\begin{figure}[h]
            \qquad
        \begin{subfigure}[b]{0.30\textwidth}
            \centering
            \includegraphics[width=\textwidth]{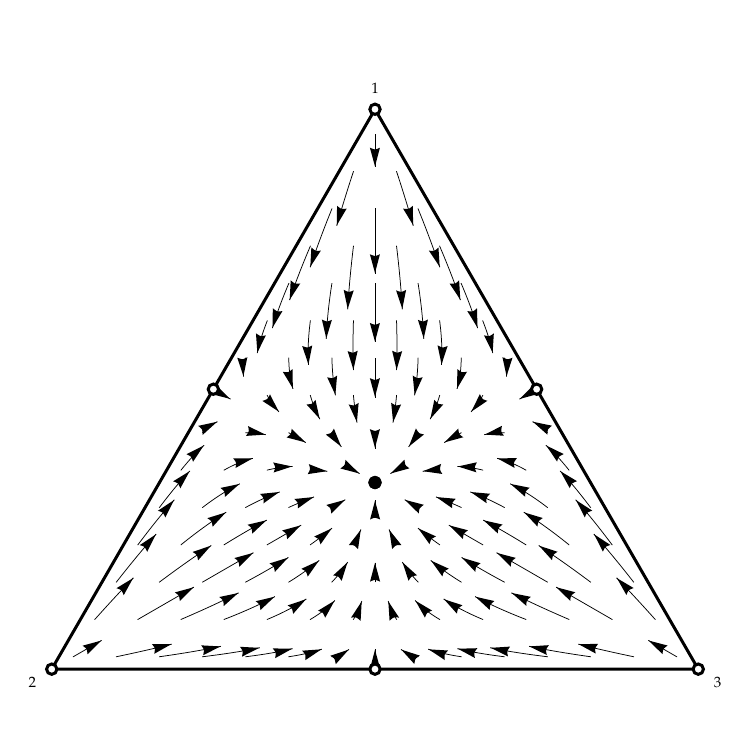}
        \end{subfigure} \qquad
        \begin{subfigure}[b]{0.4\textwidth}
            \centering
            \includegraphics[width=\textwidth]{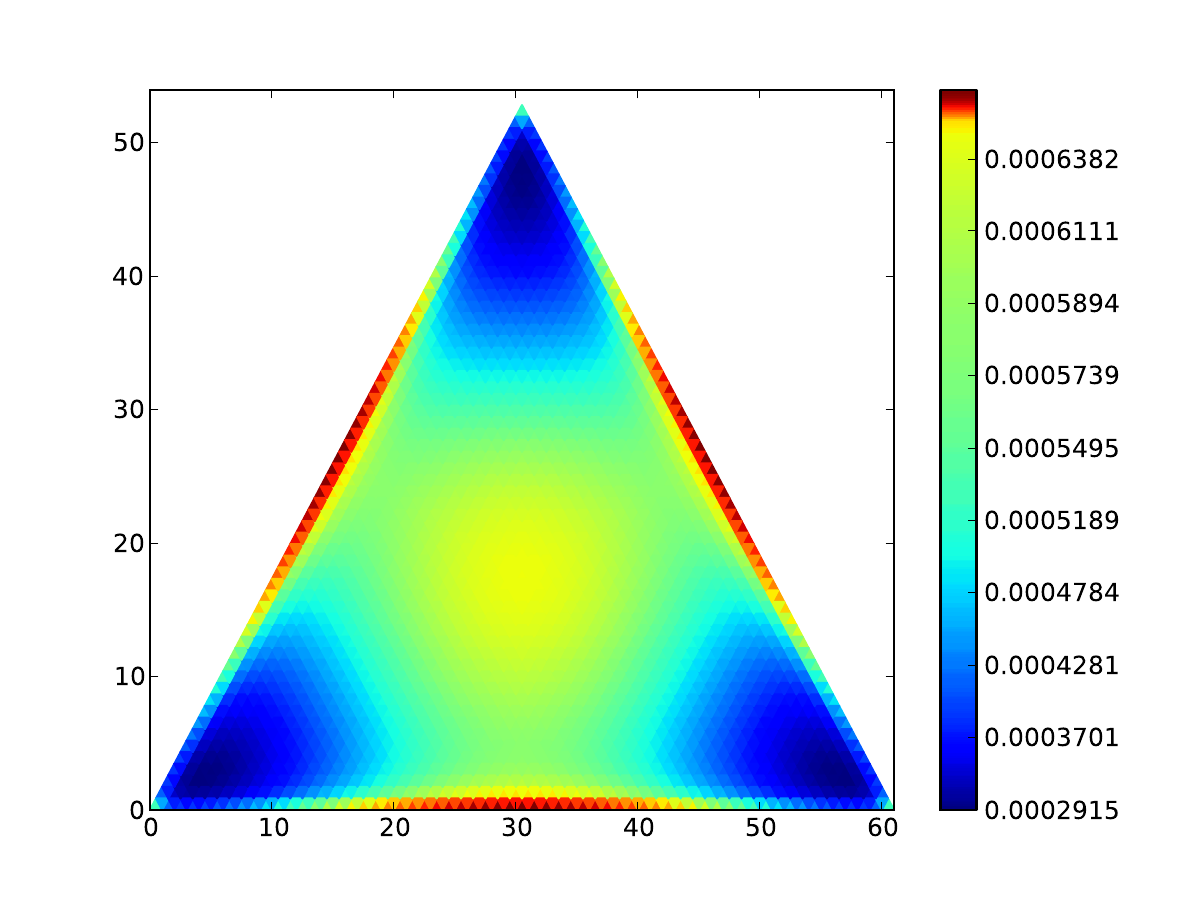}
        \end{subfigure}%
        ~ 
        \\
        \begin{subfigure}[b]{0.4\textwidth}
            \centering
            \includegraphics[width=\textwidth]{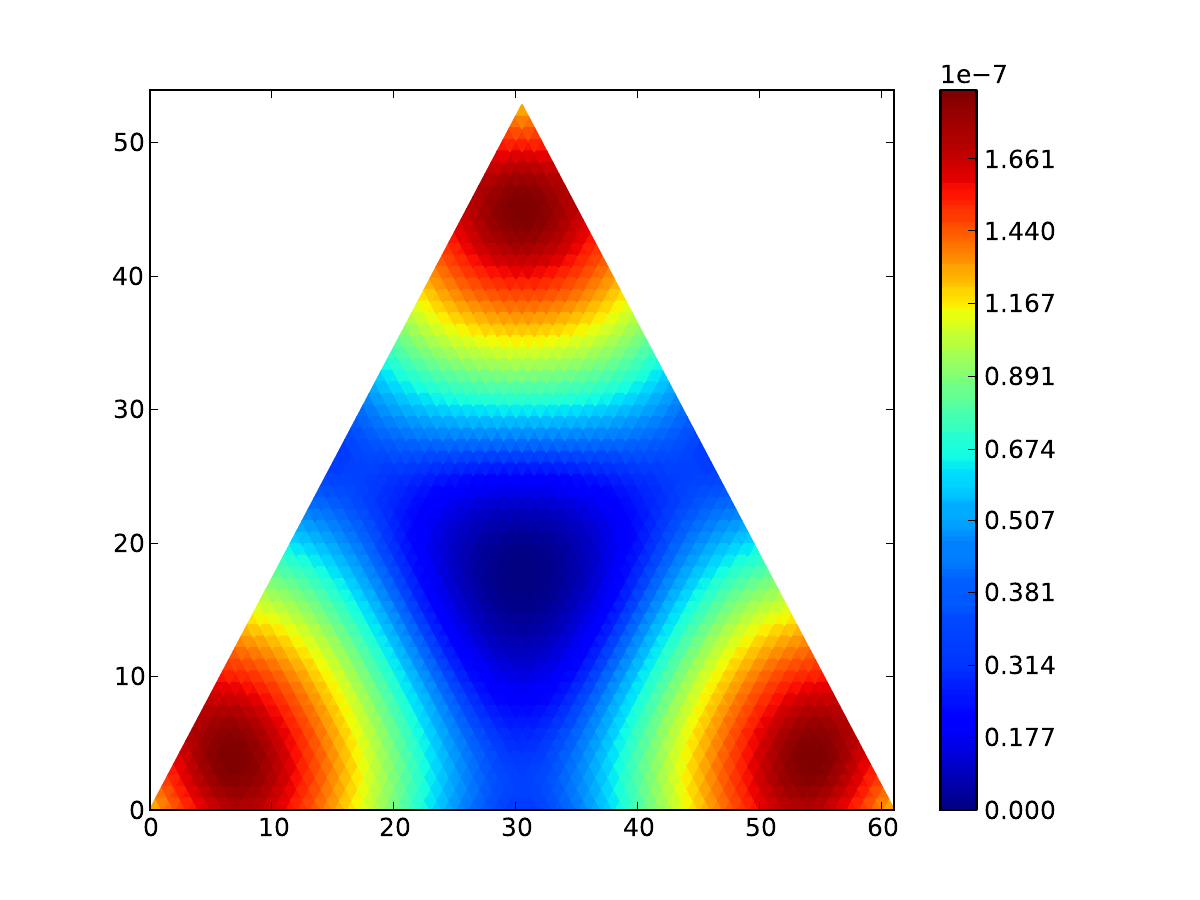}
        \end{subfigure}
        \begin{subfigure}[b]{0.4\textwidth}
            \centering
            \includegraphics[width=\textwidth]{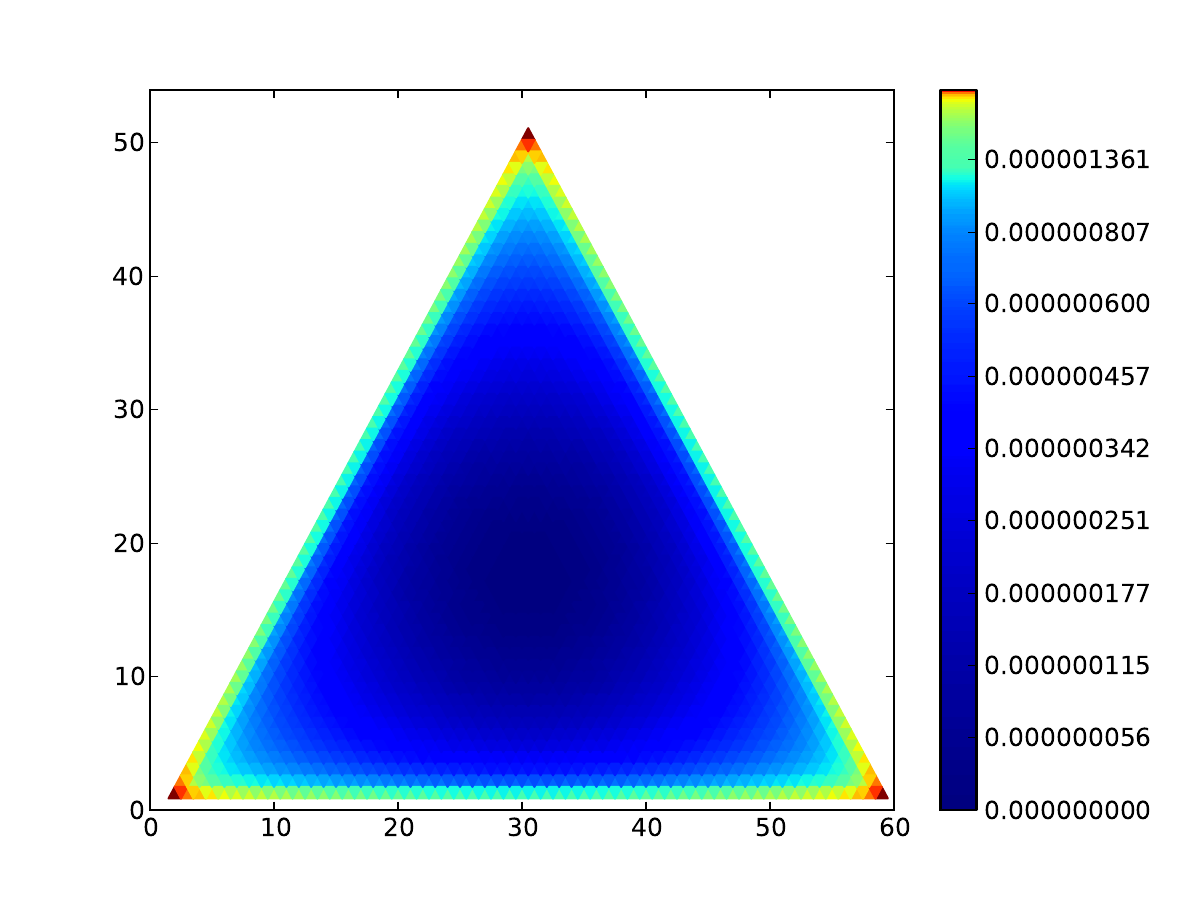}
        \end{subfigure}\\%
        ~ 
        \begin{subfigure}[b]{0.4\textwidth}
            \centering
        
        \[ A = \left( \begin{smallmatrix} 0 & 1 & 1\\
                                     1 & 0 & 1\\
                                     1 & 1 & 0
                \end{smallmatrix} \right) \]
        \end{subfigure}
        \caption{\small{Top left: Vector field for the replicator dynamic with Fermi selection, with fitness landscape defined by game matrix 7 in I.M. Bomze's classification (above); made with Dynamo \cite{franchetti2012introduction}. Top right: Stationary distribution of the Moran process with Fermi selection ($N=60$, $\beta=0.1$, $\mu=\frac{3}{2N}$) which is locally maximal at the interior stable rest point. Bottom left: Euclidean distance between each population state and the expected next state; local minima correspond to rest points of the vector field. Bottom right: Relative entropy of each state and the expected next state. For the heatmaps the boundary has not been plotted to make the interior more visible.}}
        \label{figure_graphical_abstract}
\end{figure} 

\newpage


\section{Introduction}

Finite population evolutionary dynamics are of broad interest for both application and theory \cite{taylor2004evolutionary} \cite{nowak2004emergence} \cite{ficici2000effects} \cite{fogel1998instability} \cite{muirhead2009modeling} \cite{fudenberg2006evolutionary}  \cite{nowak2006evolutionary} \cite{noble2011multivariate}; in particular, concepts of evolutionary stability have been widely-studied \cite{taylor1978evolutionary} \cite{smith1993evolution} \cite{weibull1997evolutionary} \cite{traulsen2006stochasticity}, and recently applied to finite populations \cite{wild2004fitness} with mutation (and weak selection) \cite{antal2009mutation}. Many evolutionary dynamics, such as the replicator equation, effectively assume infinite population sizes which allows a variety of tools from dynamical systems to be applied. Lyapunov stability theorems are crucial tools in the study of dynamical systems and evolutionary dynamics and there is a long history of Lyapunov stability results in evolutionary game theory and mathematical biology \cite{weibull1997evolutionary} \cite{bomze1991cross} \cite{akin1984evolutionary} \cite{sandholm2008projection} \cite{sato2005stability}. Evolutionary dynamics in finite populations are commonly modeled as Markov processes, such as the Moran and the Wright-Fisher processes \cite{moran1962statistical} \cite{moran1958random} \cite{imhof2006evolutionary} \cite{taylor2006symmetry} \cite{ewens2004mathematical}, which are not deterministic; in the case of the Moran process, the replicator equation can be recovered in a large population limit \cite{traulsen2005coevolutionary} \cite{traulsen2006coevolutionary}. Such stochastic models lack Lyapunov stability theorems. The theory of Markov processes, however, provides powerful analytic tools, notably the stationary distribution, capturing the long-run behavior of a process. We will show that there is an intimate connection between the stationary distributions of these processes and the celebrated theory of evolutionary stability.

In particular, for populations of finitely many replicating types, we define a \emph{Lyapunov-like} quantity, namely a quantity that is positive-definite, decreasing toward an ``equilibrium'' locally, and minimal at the equilibrium, for a large class of birth-death processes we call the incentive process \cite{harper2013incentive}, including the Moran process and the Fermi process \cite{traulsen2009stochastic}. We then show local maxima of the stationary distribution of the Markov processes are local minima of the Lyapunov quantity for sufficiently large populations, which in turn are generalized evolutionarily stable states. In this manner, we have effectively extended the folk theorem of evolutionary dynamics \cite{cressman2003evolutionary} \cite{hofbauer2003evolutionary} to these stochastic processes, including those with mutation (in fact we generally require some mutation for the stationary distribution to exist). Similar yet more nuanced results hold for the Wright-Fisher process on two types. For dynamics on populations of three types, such as the famous rock-scissors-paper dynamics, we show that similar results hold for the Moran process, and that the stationary distribution and Lyapunov quantity recapitulate the phase portraits of the replicator equation. Our main results apply to populations with any finite number of types. Finally we demonstrate some extensions to the dynamics of birth-death processes in structured populations, i.e. in the context of evolutionary graph theory, and of populations of variable size.

The use of information-theoretic quantities such as cross entropy and relative entropy \cite{kullback1951information} date back at least to \cite{akin1984evolutionary} and \cite{bomze1991cross}. Recently author Harper used generalizations from statistical thermodynamics \cite{naudts2011generalised} to extend the well-known Lyapunov stability result for the replicator dynamic to a large class of geometries \cite{harper2011escort}, simultaneously capturing the facts that the relative entropy is a Lyapunov function for the replicator dynamic and the Euclidean distance is a Lyapunov function for the projection dynamic \cite{sandholm2008projection}. (Local versions of such results were given first in \cite{hofbauer1990adaptive} using Riemannian geometry.) Author Fryer extended the local Lyapunov result for the relative entropy from the replicator dynamic to a class of dynamics known as the incentive dynamics \cite{fryer2012kullback}, which includes the logit, Fermi, and best-reply dynamics, and introduced the concept of an \emph{incentive stable state (ISS)} as a generalization of evolutionary stability \cite{weibull1997evolutionary}. The authors results were combined and further extended to various time-scales (including the standard discrete and continuous time-scales) in \cite{harper2012stability}, yielding a vastly general form of the classical Lyapunov stability theorem of \cite{akin1984evolutionary}. Now we extend some of these results to finite populations.

We proceed by applying a local/discrete variant of relative entropy to the study of Markov processes and drawing inspiration from methods from statistical inference. Bayesian inference is analogous to the discrete replicator dynamic \cite{shalizi2009dynamics} and relative entropy is a commonly used measure of convergence of an inference process. Simply stated, given a particular population state, we look at the expected next state of the system, formed by weighting the adjacent states by the transition probabilities of moving to those states, and compare this expected state to the current state via the standard relative entropy function (and others) from information theory. Intuitively, a population state is stable if the expected population state is close to current state. Indeed, we show that extrema of the stationary distribution minimize relative entropy between the expected state and the current state for a variety of processes. In particular we focus on the incentive process with mutation, a mapping of Fryer's incentive dynamic to finite populations, first introduced in \cite{harper2013inherent}, a generational version called the $k$-fold incentive process, and the Wright-Fisher process.

We can say more, in fact. For the distance between the expected next state and the current state, we can use any of the information-theoretic $q$-divergences \cite{harper2012stability}, which range from the Euclidean distance $(1/2)||x-y||^2$ for $q=0$ to the relative entropy for $q=1$. The Euclidean distance will give the best results for finite populations since it is well-defined on the boundary of the simplex as well as the interior, and will detect stable points everywhere. The relative entropy, however, will yield the best connection to deterministic dynamics in the limit that $N \to \infty$, where we recapture the classical results for the replicator equation.

The authors explored the incentive process without mutation in \cite{harper2013incentive}, where we investigated the fixation probabilities of the process in the absence of mutation. Together these two approaches give a complete characterization of the behavior of the incentive process, whether the ``equilibria'' of the process are the boundary fixation states or if there are interior maxima of the stationary distribution. Stationary distributions of the Moran process and variants with mutation have been studied by Claussen and Traulsen \cite{claussen2005non} and others e.g. \cite{taylor2004evolutionary}. We note that the incentive processes captures a similar set of processes as Sandholm's microfoundations approach using revisions protocols \cite{sandholm2005excess} \cite{sandholm2010stochastic} \cite{sandholm2012stochastic}.

\section{Preliminaries}

%

\subsection{Incentive Proportionate Selection with Mutation}

The incentive process was briefly introduced in the appendix of \cite{harper2013inherent}, as a generalization of the Moran process incorporating the concept of an incentive. An incentive is a function that mediates the interaction of the population with the fitness landscape, describing the mechanisms of many dynamics via a functional parameter, including replicator and projection dynamics, and logit and Fermi processes. A Fermi incentive is frequently used to avoid the general issue of dividing by zero when computing fitness proportionate selection \cite{traulsen2007pairwise} \cite{traulsen2009stochastic}. The authors described the fixation probabilities of the incentive process without mutation in \cite{harper2013incentive}. We now describe the incentive process with mutation, which captures a variety of existing processes, such as those used in \cite{claussen2005non} and \cite{taylor2004evolutionary}.

Let a population be composed of $n$ types $A_1, \ldots A_n$ of size $N$ with $a_i$ individuals of type $A_i$ so that $N = a_1 + \cdots + a_n$. We will denote a population state by the tuple $a = (a_1, \ldots, a_n)$ and the population distribution by $\bar{a} = a / N$. Define a matrix of mutations $M$ where $0 \leq M_{i j} \leq 1$ may be a function of the population state for our most general results, but we will typically assume in examples that for some constant value $\mu$, the mutation matrix takes the form $M_{i j} = \mu / (n-1)$ for $i \neq j$ and $M_{i i} = 1 - \mu$. Finally we assume that we have a function $\incentive(\bar{a}) = (\incentive_1(\bar{a}), \ldots, \incentive_n(\bar{a}))$ which takes the place of $a_i f_i(a)$ in the Moran process. See Table \ref{incentives_table} for a variety of example incentives. We will denote the normalized distribution derived from the incentive function as $\bar{\incentive}$.

The incentive process is a Markov process on the population states defined by the following transition probabilities, corresponding to a birth-death process where birth is incentive-proportionate with mutation and death is uniformly random. To define the adjacent population states, let $i_{\alpha \beta}$ be the vector that is 1 at index $\alpha$, -1 at index $\beta$, and zero otherwise, with the convention that $i_{\alpha \alpha}$ is the zero vector of length $n$. Every adjacent state of state $a$ for the incentive process is of the form $a + i_{\alpha \beta}$ for some $1 \leq \alpha, \beta \leq n$. At a population state $a$ we choose an individual of type $A_i$ to reproduce proportionally to its incentive, allowing for mutation of the new individual as given by the mutation probabilities. The distribution of incentive proportionate selection probabilities is given by $p(\bar{a}) = M(\bar{a}) \bar{\incentive}(\bar{a})$; explicitly, the $i$-th component is
\begin{equation}
p_i(\bar{a}) = \frac{\sum_{k=1}^{n}{\incentive_k(\bar{a}) M_{k i} }}{\sum_{k=1}^{n}{\incentive_k(\bar{a})}}
\label{incentive-proportionate-reproduction} 
\end{equation}
We also randomly choose an individual to be replaced, just as in the Moran process. This yields the transition probabilities
\begin{align}
T_{a}^{a + i_{\alpha, \beta}} &= p_{\alpha}(\bar{a}) \bar{a}_{\beta} \qquad \text{for $\alpha \neq \beta$} \notag \\
T_{a}^{a} &= 1 - \sum_{b \text{ adj } a, b\neq a}{T_{a}^{b}}
\label{incentive_process}
\end{align}

We will mainly consider incentives that are based on fitness landscapes of the form $f(\bar{a}) = A \bar{a}$ for a game matrix $A$; some authors do not allow self-interaction and use fitness landscapes of the form (e.g.):
\begin{align*}
f_1(i) &= \frac{a(i-1) + b(N-i)}{N-1} \\
f_2(i) &= \frac{ci + d(N-i-1)}{N-1}
\end{align*}
for a game matrix defined by
\[ A = \left( \begin{smallmatrix}
 a & b\\
 c & d
\end{smallmatrix} \right) \]

For our purposes the difference will have little impact. Though we primarily investigate two one-parameter families of incentives defined in terms of a fitness landscape, we note that the incentive need not depend on a fitness landscape or the population state at all. The Fermi process of Traulsen et al is the $q$-Fermi for $q=1$ \cite{traulsen2009stochastic}, and $q=0$ is called the logit incentive, which is used in e.g. \cite{andrae2010entropy}. The $q$-replicator incentive has previously been studied in the context of evolutionary game theory \cite{harper2011escort} \cite{harper2012stability} and derives from statistical-thermodynamic and information-theoretic quantities \cite{tsallis1988possible}. Recently human population growth in Spain has been shown to follow patterns of exponential growth with scale-factors $q \neq 1$ \cite{hernando2013workings}. Other than the best-reply incentive, which we will consider primarily as a limiting case of the $q$-Fermi, we will assume that all incentives are positive definite $\incentive(k) > 0$ if $0 < k < N$ to avoid the restatement of trivial hypotheses in the results that follow. For the $q$-Fermi incentives, positivity is of course guaranteed for all landscapes. This is particularly convenient for three-type dynamics since the mean-fitness can frequently be zero (for zero-sum games such as the rock-paper-scissors game), which would cause the transition probabilities to be ill-defined. We will also typically assume that incentives are continuous in the population distribution, that is, that they are discretized versions of continuous functions on the probability simplex.

At first glance the introduction of incentives might seems like just an alteration of the fitness landscape. While it is true that on the interior of the probability simply many incentives can be re-written as nonlinear fitness functions, there are significant advantages to the change in perspective from fitness-proportionate selection to incentive proportionate-reproduction. In particular, the authors showed in \cite{harper2012stability} that incentives ultimately lead to a deeper understanding of evolutionary stability and a substantial improvement in the ability to find Lyapunov functions (via formal and geometric considerations) for a wide-range of evolutionary dynamics. Moreover, when dynamics interact with the boundary, such as innovative and non-forward invariant dynamics, incentives yield a better description of evolutionary stability. Finally, we can study how a population interacts with the same fitness landscape by varying the incentive function and its parameters rather than the fitness landscape itself.

\begin{figure}[h]
    \centering
    \begin{tabular}{|c|c|}
        \hline
        Dynamics & Incentive \\ \hline
        Projection & $\incentive_i(\bar{a}) = f_i(\bar{a})$\\ \hline
        Replicator & $\incentive_i(\bar{a}) = \bar{a}_i f_i(\bar{a})$\\ \hline
        $q$-Replicator & $\incentive_i(\bar{a}) = \bar{a}_i^{q} f_i(\bar{a})$\\ \hline
        $q$-Fermi & $\displaystyle{ \incentive_i(\bar{a}) = \frac{\bar{a}_i^{q} \text{exp}(\beta f_i(\bar{a}))}{\sum_{j}{\bar{a}_j^{q}\text{exp}(\beta f_j(\bar{a}))}}}$ \\ \hline
        Best Reply & $\incentive_i(\bar{a}) = \bar{a}_i BR_i(\bar{a})$ \\ \hline
    \end{tabular}
    \caption{Incentives for some common dynamics. The projection incentive \cite{sandholm2008projection} is simply the $q$-replicator with $q=0$. Another way of looking at the projection incentive on the Moran landscape is simply as a \emph{constant incentive}, since the incentive itself is a constant function in $i$. The logit incentive is the $q$-Fermi with $q=0$. For more examples see Table 1 in \cite{fryer2012existence}.}
    \label{incentives_table}
\end{figure}

\subsection{The Wright-Fisher Process}

In contrast to the Moran process, which models a population in terms of individual birth-death events, the Wright-Fisher process is a generational model of evolution \cite{imhof2006evolutionary} \cite{ewens2004mathematical}. Each successive generation is formed by sampling, proportionally to incentive, the current generation. Define the Wright-Fisher Process with mutation for evolutionary games by the following multinomial transition probabilities:
\begin{equation} \label{wright_fisher_process}
T_{a}^{b} = \binom{N}{b} \prod_{i}{ p_i(\bar{a})^{b_i}},
\end{equation}
where $p_i(\bar{a})$ is defined as before in Equation \ref{incentive-proportionate-reproduction}. This is a slight generalization of the basic process as given by Imhof and Nowak to include mutation, though we will not consider parameters for intensity of selection via the $1-w$, $w$ linear combination approach as in \cite{imhof2006evolutionary}. In contrast to the Moran process, the Wright-Fisher process is not tri-diagonal, rather every state is accessible from every other state, so long as the incentive-proportionate probabilities are non-zero.

\subsection{$n$-fold Incentive process}

Since the Wright-Fisher process is a generational process, replacing the entire population in each iteration, and the incentive process is atomic process, they can exhibit very different behaviors. Consider the following process, which will be referred to as the $k$-fold incentive process, as an intermediate between the two types of processes. Define each step of the process as $k$ steps of the Moran process, so that $k=1$ is the Moran process, and $k=N$, where $N$ is the population size, yields a \emph{generational} processes derived from the incentive process. The crucial difference is the simultaneity of the replication events, as all occur instantaneously for the Wright-Fisher process, so the transition probabilities for the $k$-fold process differ in general.

The transition probabilities of the $k$-fold process can be computed directly from the transition matrix of the incentive process (equation \ref{incentive_process}) by simply computing the $k$-th power of the transition matrix. The entries of the transition matrix, $\left(T^k\right)_{a}^{b}$ correspond to the probability of moving from population state $a$ to population state $b$ in exactly $k$ steps of the Moran process.

\subsection{Stationary Distributions}

For a deterministic dynamic there are notions of long-term behavior, such as an $\omega$-limit set\cite{hofbauer1998evolutionary}, and powerful tools such as the Lyapunov stability theorems. For Markov processes we lack any notion of deterministic behavior since the population trajectories can only be determined probabilistically from previous states. We do, however, have a tool that does not exist in the deterministic setting: the stationary distribution of the Markov process. The stationary distribution is a probability distribution on the states of the process, indicating the likelihood of finding the population in any particular state in the long-run. It is characterized in multiple ways, of which we will use that facts that (1) the stationary distribution $s$ is the left eigenvector of the transition matrix with eigenvalue one, that is $s = sT$, and (2) the rows of the limit $T^k x$ convergence to the stationary distribution for any starting distribution $x$ as $k \to \infty$. 

Stationary distributions for the Moran process in two dimensions and some recently-studied generalizations are given in several recent works \cite{claussen2005non} \cite{antal2009strategy}  \cite{taylor2004evolutionary} \cite{khare2009rates} \cite{gokhale2011strategy} \cite{allen2012measures} \cite{wu2013dynamic}; stationary distributions for both the Moran process and the Wright-Fisher process have been studied in \cite{huillet2010discrete}; a number of formulas for various finite population processes are given in \cite{sandholm2007simple} and \cite{sandholm2014large}. The $n=2$ solution only relies on the fact that the transition matrix is tridiagonal, and so applies to the incentive process with mutation without modification. The components $s_a$ of the stationary distribution satisfy the detailed-balance condition $s_a T_{a}^{b} = s_{b} T_{b}^{a}$ (also known as the local balance condition) and such processes have a stationary distribution given by the following formula. For any sequence of states $j_0, j_1, \ldots, j_k$,
\begin{equation}
s_{j_k} = s_{j_0} \prod_{i=0}^{k-1}{\frac{T_{j_i}^{j_{i+1}}}{T_{j_{i+1}}^{j_{i}}}}
\label{s_j}
\end{equation}
where $s_{j_0}$ can be obtained from the normalization $\sum_i{s_{i}} = 1$.
This particular formulation relies on the fact that the transitions between neighboring states are nonzero, which we assume is valid throughout. For games with more than two types, none of the processes described so far satisfy the detailed-balance condition in general \cite{taylor2006symmetry}, though the neutral landscape is reversible for some choices of mutation matrix. It is, however, always the case that the global balance conditions are satisfied:
\begin{equation}
s_a \sum_{a \text{ adj } b}{T_{a}^{b}} = \sum_{b \text{ adj } a} {s_{b} T_{b}^{a}},
\label{global_balance_equation}
\end{equation}
which follows easily from the fact that the stationary distribution is the right eigenvector of the transition probability matrix $s T = s$. Though we will not be able to give explicit closed forms in the higher dimensional cases (in general they are not known), the global balance equations will be sufficient for our analytic purposes. Exact analytic forms for the stationary distribution of the Wright-Fisher process are similarly not known. Nevertheless, from a computational perspective, for any concrete values of the various parameters of these processes, the stationary distribution can be computed fairly efficiently for relatively large populations.

%
%

\section{Notions of Stability}
Let us pursue notions of stability for these Markov processes. Intuitively, a population state is stable if there is a reluctance of the population to move from that state. We characterize this in multiples ways. Evolutionary stability has been studied extensively and is closely related to the concept of a Nash equilibrium, which can be loosely described as a state in which no strategy has incentive to deviate from. In \cite{fryer2012existence}, Fryer defines an incentive stable state (ISS) as a generalization of evolutionary stability, and shows that there is a generalization of the well-known Lyapunov theorem for the replicator equation. The authors ported the notion of an incentive stable state to the incentive process in \cite{harper2013incentive}. In some cases an ISS for a particular incentive is again an ESS, such as for the best-reply incentive, and the Fermi incentive. We will give a generalization of those definitions in this work for completeness.

From the Moran process, an analogous form for the standard equation for an interior evolutionary stable state (ESS) for a 2x2 game can be obtained by equating the fitness functions of the two types in the population, matching the fact that for the replicator equation any interior rest point has equal fitness across all types \cite{nowak2006evolutionary}. This can similarly be obtained by equating the transition probabilities $T_{(i, N-i)}^{(i+1, N-i-1)} = T_{(i, N-i)}^{(i-1, N-i+1)}$ for the Moran process, which also reproduces the correct form of Fryer's ISS in equality for the incentive process, and adds $i=0$ and $i=N$ to set of potential stable points (for the Moran process). Since we do not (yet) have inequality in a neighborhood (as in the definition of a standard ESS) from this derivation we will refer to such states as \emph{candidate ISS}. Explicitly, we define a \emph{candidate ISS} as a population state $a$ such that $\bar{a} = p(\bar{a})$. For $\mu = 0$ and the replicator incentive, this definition implies that $f_i(\bar{a}) = f_j(\bar{a})$ for all $i$ and $j$, which is typical for an ESS. Note that solutions of the ISS candidate equation in a finite population will often not be integral.


Although we lack a Lyapunov theory for Markov processes, it is nevertheless desirable to have a quantity that is locally minimal at stable states and decreasing locally otherwise as an indicator of stability. Such a function yields another method of computation of candidate ISS without having to solve the ISS candidate equation, which might be very difficult. We combine this desire with the idea that a stable state for a stochastic process should \emph{remain close} to itself by defining the \emph{expected distribution} resulting from a single iteration of the incentive process
\[E(\bar{a}) = \frac{1}{N}\sum_{b \text{ adj a}}{b T_{a}^{b}}. \]
We then consider distance functions such as $D_1(a) = D_{KL}(E(\bar{a}) || \bar{a})$ where $D_{KL}(x, y) = \sum_{i}{x_i \log{(x_i / y_i)}}$ is the relative entropy or Kullback-Leibler divergence, a commonly-used measure of distance between probability distributions.

For this purpose, we define the following one-parameter distance function
\begin{equation} \hat{D}_d(x, y) = \begin{cases}
\frac{1}{2}||x - y||^2 & \text{if } d=0 \\
D_{KL}(x, y) & \text{if } d=1 \\
\frac{1}{1-d}\sum_{i}{\left[ \frac{x_i^{2-d} - y_i^{2-d}}{2-d} - y_i^{1-d}(x_i-y_i)\right]} & \text{if } 0 < d < 1 \\
\end{cases}
\label{distance_function}
\end{equation}
This family has the property that $\hat{D}_{d_1} < \hat{D}_{d_2}$ if $d_1 < d_2$, which is an simple consequence of a more general definition (see \cite{harper2011escort}), and all have the property that $\hat{D}_d(x,y) = 0$ iff $x=y$ (positive definite). We will also make use of the fact that they are all bounded above (on our state spaces) by the $\chi$-squared distance \cite{sayyareh2011new}
\[ D_{\chi^2}(x, y) = \sum_{i}{\frac{\left(x_i - y_i\right)^2}{x_i}}.\]

In what follows, we will write $D_d(a)$ for $\hat{D}_d(E(\bar{a}), \bar{a})$. This brings us to our first propositions. The first gives two forms for the expected distribution of the incentive, both of which are simple algebraic consequences of the definitions given so far. For two-type populations, the second form is, explicitly,

\[ E(\bar{a}) = \bar{a} + \frac{1}{N} \left(T_{(i, N-i)}^{(i+1, N-i-1)} - T_{(i, N-i)}^{(i-1, N-i+1)}\right) \left[ \begin{matrix} 1 \\ -1 \end{matrix} \right] \]

\begin{proposition}
For the incentive process, we have two forms for the expected distribution:
\begin{enumerate}
 \item \ds{ E(\bar{a}) = \frac{N-1}{N} \bar{a} + \frac{1}{N} p(\bar{a})} \\
 \item \ds{ E(\bar{a}) = \bar{a} + \frac{1}{N} \sum_{\alpha, \beta, \alpha \neq \beta}{i_{\alpha \beta} T_{a}^{a + i_{\alpha \beta}}} = \bar{a} + \frac{1}{N} \sum_{\alpha, \beta, \alpha < \beta}{i_{\alpha \beta} \left(T_{a}^{a + i_{\alpha \beta}} - T_{a}^{a - i_{\alpha \beta}} \right)  }  }
\end{enumerate}
\label{expected_incentive}
\end{proposition}

\begin{proposition}
For both the incentive process and the Wright-Fisher process, the following are equivalent:
\begin{enumerate}
 \item $\bar{a} = p(\bar{a})$ (ISS Candidate definition)
 \item $E(\bar{a}) = \bar{a}$, that is, $\bar{a}$ is a fixed point of $E$
 \item $D_d(a) = 0$ for all $0 \leq d \leq 1$
\end{enumerate}
\label{prop_expected_next}
\end{proposition}
\begin{proof}
(2) For Wright-Fisher follows from the fact that the expected value of a multinomial distribution is just $p(\bar{a})$. The others follow easily from the previous proposition and Equation \ref{incentive-proportionate-reproduction}.
\end{proof}

For the $k$-fold incentive process the natural notion of expected distribution $E^k(a)$ is the $k$-th iterate of $E$ as defined for the incentive process. We will not give a closed form for $E^k$; the following proposition will suffice for our purposes, and is also an easy consequence of our definitions.

\begin{proposition}
Let $E(\bar{a})$ be defined as for the incentive process. Then for the $k$-fold incentive process, $E^k(\bar{a}) = \bar{a}$ if and only if $E(\bar{a}) = \bar{a}$.
\end{proposition}

Hence for all three processes we have that ISS candidates are equivalent to fixed points of the expected distribution function. This definition captures the usual notion of ISS/ESS as follows. We can write (loosely) that $\bar{a} = p(\bar{a}) \approx \bar{\incentive}(\bar{a}) (1 - \mu) + \mu / (n-1) \mathbf{1}$, where $n$ is the number of types in the population, so that if e.g. $\mu \sim O(1/N)$ or is otherwise relatively small, we have asymptotically that $\bar{a} \approx \bar{\incentive}(\bar{a})$, i.e. a fixed point of $\bar{\incentive}$ (and so a candidate for incentive stability). For the replicator incentive, this implies that $f_i(\bar{a}) = f_j(\bar{a})$ for all $i,j$, which is the case for an ESS, as above. Our results are true regardless of the form of the mutation matrix, but the stable points need not coincide with the ISS of the incentive for large or imbalanced mutations (which is a feature). We note that Garcia et al have found that the form of the mutations can significantly alter to locations of stable points \cite{garcia2012structure}.

The stationary distribution yields our final notion of stability. We define a candidate ISS to be an ISS if the stationary distribution at the candidate is locally maximal, and we call such local maxima \emph{stationary stable} in general. Intuitively, a stable state of the dynamic should occur at some notion of a maximum of the stationary distribution since it indicates the likelihood of finding the process in any particular state, in the long run, and a tendency to remain in the state. Our main results are that in many cases, for sufficiently large populations, local maxima and minima of the stationary distribution are ISS candidates.

We note that several authors (e.g. \cite{antal2009strategy} \cite{gokhale2011strategy} \cite{allen2012measures} \cite{wu2013dynamic}) have used the stationary distribution as a solution criterion, looking at average abundance of each population type computed by weighting the population states by the stationary distribution. While the authors of this paper find this to be an interesting approach, we use the stationary distribution to measure the stability of particular population states rather than as a measure on individual types. To see how these differ, consider that for any symmetric process on two types, the average abundance is equal for the two types \cite{antal2009strategy}. For instance, the neutral landscape and the landscape given by $a=1=d$, $b=2=c$ for the replicator incentive would have both have equal abundance for the two types and a stationary distribution with maximum at $(N/2, N/2)$, but so would the neutral landscape as $\mu \to 0$, in which the stationary distribution would be locally maximal at the two fixation states $(0,N)$ and $(N,0)$ \cite{fudenberg2004stochastic} and zero at the central point. In all these cases strategy abundance gives an intuitive aspect of ``equally-likely to be represented'' in the long run, yet taking the average loses information about the manner in which this representation manifests (coexistence of all types or domination by a single type). In all cases the given states are stationary stable.

\begin{figure}[h]
        \begin{subfigure}[b]{0.4\textwidth}
            \centering
            \includegraphics[width=\textwidth]{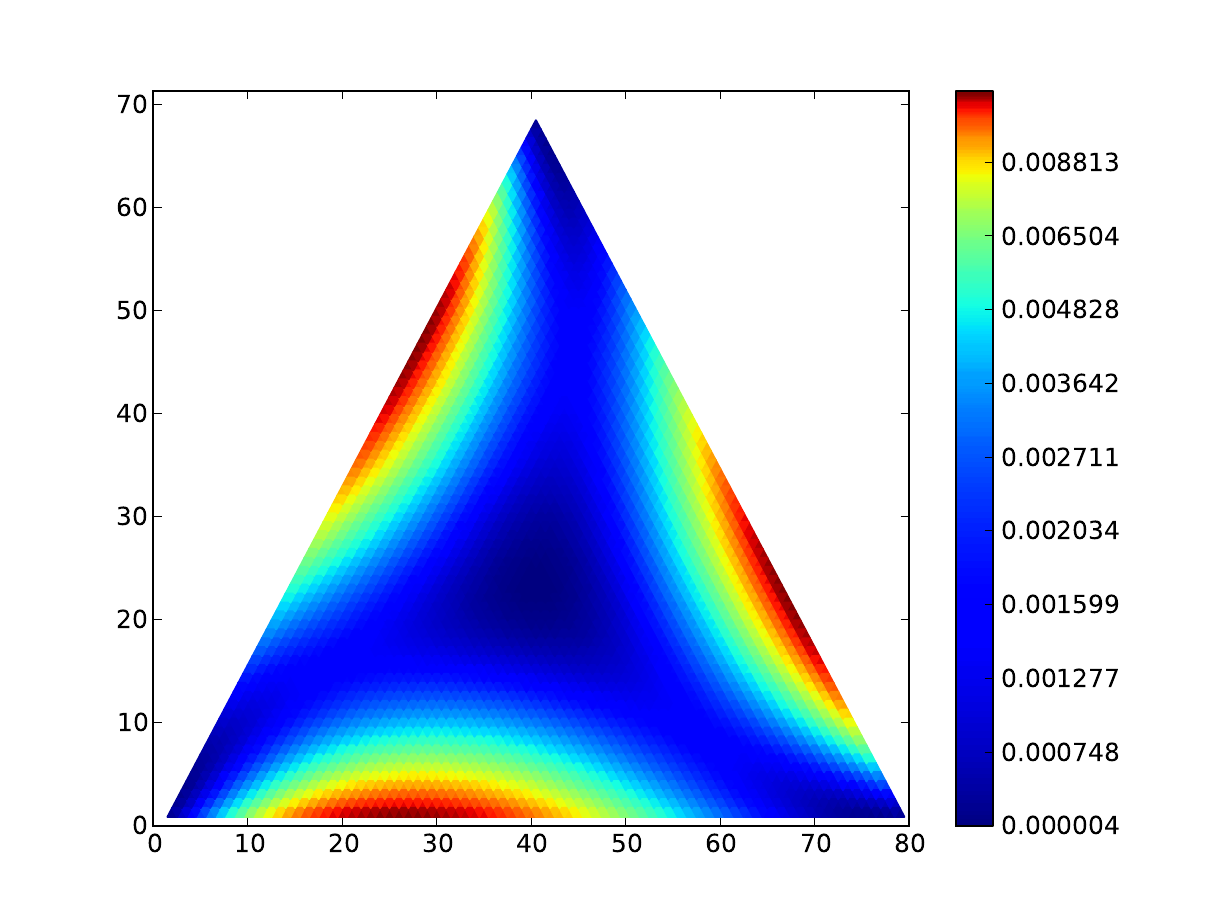}
        \end{subfigure}
        \begin{subfigure}[b]{0.4\textwidth}
            \centering
            \includegraphics[width=\textwidth]{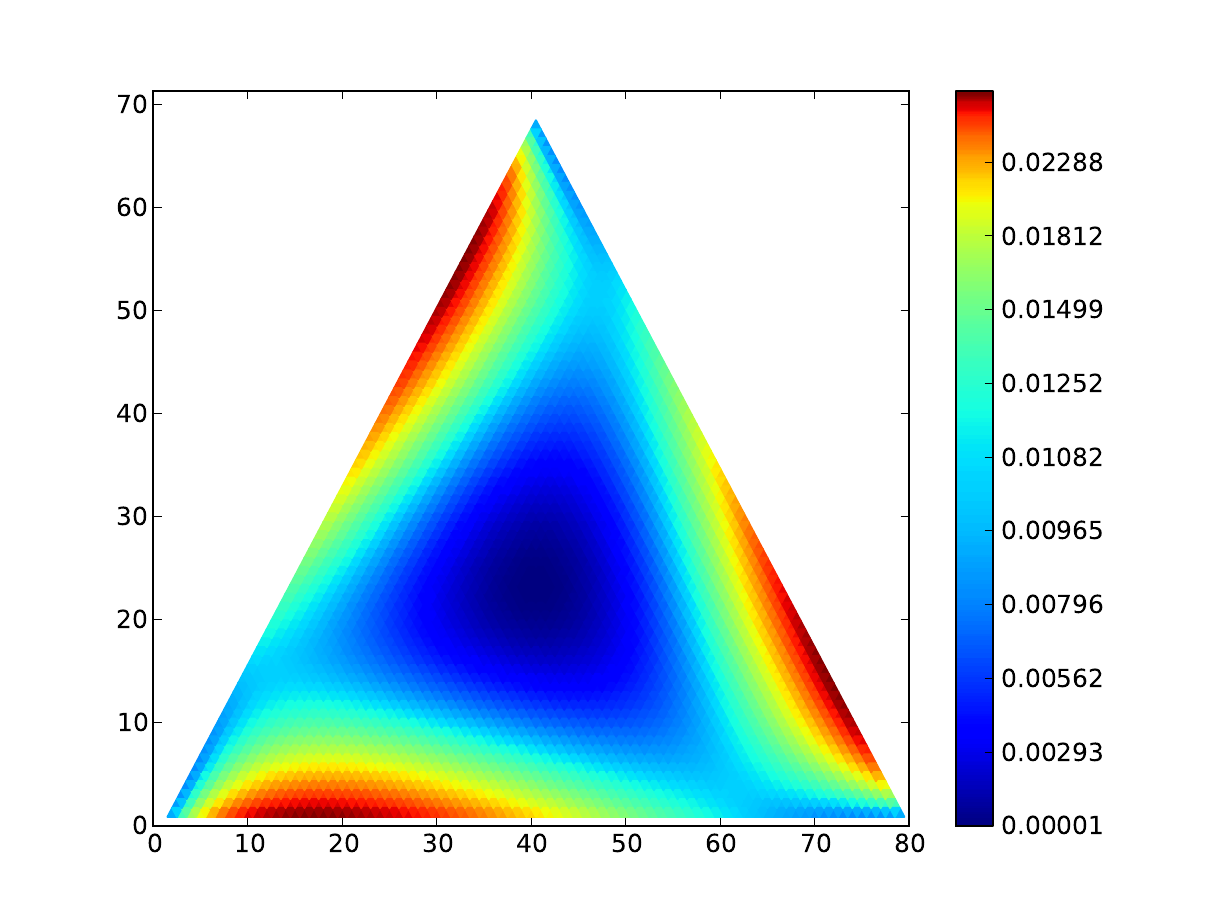}
        \end{subfigure}%
        \caption{Left: $D_0$ for the $40$-fold incentive process, Fermi incentive ($\beta=1$), Rock-Paper-Scissors matrix (Bomze's 17), $N=80$. Local minima occur at the boundary states and at the center. Right: $D_1$ for the same process. Compare to Figure 1.4 in \cite{traulsen2009stochastic}, which depicts a Lyapunov quantity equivalent to the relative entropy for the replicator equation.}
        \label{figure_rsp_lyapunov}
\end{figure} 

\section{Main Theorem for Two-type Incentive Processes}

First we show that when the stationary distribution is locally maximal or minimal, the distance between $E(\bar{a})$ and $\bar{a}$ is minimal. We require a technical lemma to prove the main theorem that concerns continuity of incentive-proportionate replication probabilities. The crux of the proof is essentially that from the stationary distribution we have that if $(i, N-i)$ is local maximum or minimum of the stationary distribution then
\[ T_{(i, N-i)}^{(i+1, N-i-1)} = T_{(i+1, N-i-1)}^{(i, N-i)}. \]
Note there are two directions for the path through the states in Equation \ref{s_j}, ending with $i \to i+1$ or $i \to i-1$, so there is a second solution to the ISS max/min criteria as well (see Table \label{iss_solutions}). To satisfy $E(\bar{a}) = \bar{a}$ we need 
\[ T_{(i, N-i)}^{(i+1, N-i-1)} = T_{(i, N-i)}^{(i-1, N-i+1)}.\]
So we need ``equality'' of the right-hand sides to ensure that these criterion are both satisfied, and we can argue that this is the case for sufficiently large $N$ by appealing to continuity. This should not be surprising -- in both cases, the population is shifting by one in the same direction, and as $N$ gets larger, these points are closer together since $i/N \approx (i+1)/N$ for large $N$.

\begin{lemma}
Suppose the vector of incentive proportionate selection probabilities $p$ as defined in Equation \ref{incentive-proportionate-reproduction} is a continuous function of $\bar{a}$ and let $\epsilon > 0$. Then for the incentive process defined in Equation \ref{incentive_process}, there exists an integer $N'$ such that for $N > N'$ we have that \[\left|T_{(i+1, N-i-1)}^{(i, N-i)} - T_{(i, N-i)}^{(i-1, N-i+1)}\right| < \epsilon.\]
\end{lemma}
\begin{proof}
Since $p$ is a continuous function of $\bar{a}$, so is $p^{*}(\bar{a}) = (1 - p(\bar{a})) * \bar{a}$. Let $\bar{a}^+=\bar{a} + i_{1,2}/2$ and $\bar{a}^- = \bar{a} + i_{2,1}/2$. For $\epsilon > 0$ we simply need to choose $N$ large enough so that the set of population states is sufficiently fine-grained enough so that $|T_{\bar{a}^+}^{\bar{a}} - T_{\bar{a}}^{\bar{a}^-}| = \left|p^*(\bar{a}^+) - p^*(\bar{a})\right| < \epsilon$. Continuous functions on compact spaces (the simplex) are uniformly continuous, so $p^*$ continuous implies there is $\delta > 0$ such that if $|\bar{a} - \bar{b}| < \delta$ then $\left|p^*(\bar{a}) - p^*(\bar{b})\right| < \epsilon$ for all $a$ and $b$. Since $|\bar{a}^+ - \bar{a}| = 2 / N$, $N > N' = 2 / \delta$ is sufficient.
\end{proof}

If the incentive function is continuous as a function of $\bar{a}$, and the mutations are constant or otherwise continuous as functions of $a$, then we have that $p(\bar{a})$ is also continuous, satisfying the hypothesis of the lemma. 

\begin{theorem}
For a continuous incentive and mutation matrix, there is a sufficiently large $N$ so that local maxima and minima of the stationary distribution of the incentive process are local minima of the distance $D_d$ for $0 \leq d < 1$; interior extrema are local minima of the relative entropy $D_1$.
\label{main_thm_local}
\end{theorem}

\begin{proof}
If $(j, N-j)$ is a maximum or minimum of the stationary distribution then by the formula for the stationary distribution (Equation \ref{s_j}) we must have that the last term in the product passes through 1 so that the stationary distribution switches from increasing to decreasing (or vice versa). By continuity there is $i \in (j-1, j+1)$ where the ratio is exactly equal to one by the intermediate value theorem, and so the state $(i, N-i)$ satisfies $T_{(i, N-i)}^{(i+1, N-i-1)} = T_{(i+1, N-i-1)}^{(i, N-i)}$.

Let $\epsilon > 0$. Then by the lemma and the triangle inequality, for sufficiently large $N > N'$ (scaling $(i, N-i)$ if necessary),
\[ \left|T_{(i, N-i)}^{(i+1, N-i-1)} - T_{(i, N-i)}^{(i-1, N-i+1)} \right| \leq \left| T_{(i, N-i)}^{(i+1, N-i-1)} - T_{(i+1, N-i-1)}^{(i, N-i)}\right| < \epsilon \]
For the $\chi^2$-distance, we have that, by Proposition 2, (at interior states)
\[D_{\chi^2}(a) < N \left| T_{(i, N-i)}^{(i+1, N-i-1)} - T_{(i+1, N-i-1)}^{(i, N-i)}\right|^2 < \frac{2}{N}\epsilon^2 ,\]
and similarly at boundary states for $D_d$, $0 < d < 1$. This shows that $D_d$ vanishes for large $N$; at neighboring states, the difference in transitions is larger, and $D_d(a)$ is locally minimal.
\end{proof}

A maximum of the stationary distribution need not occur when $T_{i \to i}$ is maximal, hence it is not a good definition of stability. The converse of the theorem is not true; the Lyapunov quantity may be minimized but fail to have a local maximum or minimum of the stationary distribution. While we have only proven that the theorem for large $N$ for arbitrary continuous incentives, it may be possible to make sharper estimates based on the form of the incentive. As we will see from the examples for dimensions greater than 2, for linear landscapes the theorem seems to hold for almost any $N$ with few exceptions.

For the replicator equation, it is known that there is at most one asymptotically stable interior state for linear fitness landscapes. We have a similar result in this context as a corollary.

\begin{corollary}
For two-type populations, linear fitness landscapes, the Moran process, if $\mu \to 0$ as $N \to \infty$ then there is at most one interior stationary stable state as $N \to \infty$.
\end{corollary}

Note that the hypothesis that $\mu \to 0$ is necessary, otherwise one can find counterexamples with multiple interior stationary maximum for specific parameters, e.g. $\mu = 6/25$, and game matrix given by $a=1 = d$, $b=0=c$ produces a process with interior local maxima at $a/N = 2/5, 3/5$. In fact for this game, any $\mu \in (0, 1/4)$ produces two interior solutions, $a/N = 1/2 \pm \sqrt{1-4m}/{2}$, so long as $N$ is large enough for the solutions to be more than $1/N$ from the boundaries. (Note that Theorem \ref{main_thm_local} still holds.)

\begin{figure}[h]
    \centering
    \includegraphics[width=0.5 \textwidth]{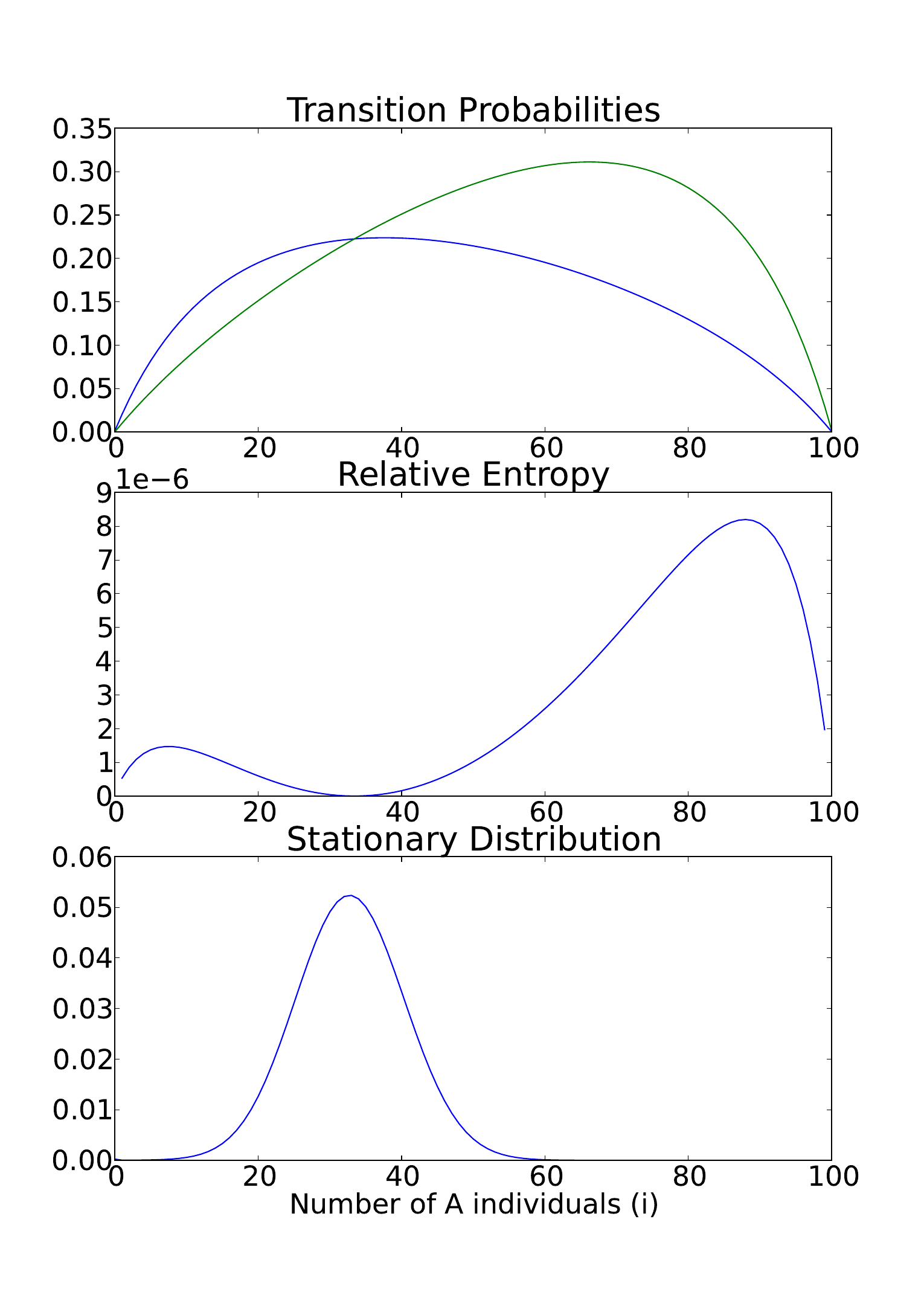}
    \caption{Demonstration of Theorem \ref{main_thm_local}, for the Moran process (replicator incentive) with game matrix given by $a=1=d$, $b=2$, $c=3$, population size $N=100$, $\mu=1/1000$, which has an ESS at approximately $(33, 67)$. Top: Transition probabilities $T_{i}^{i+1}$ (blue) and $T_{i}^{i-1}$ (green). Middle: Relative entropy $D(\bar{a})$. Bottom: Stationary distribution.}
    \label{figure_2}
\end{figure}

\subsection{$k$-fold Incentive Process}

These two propositions show that the main theorem for locally-balanced birth-death processes will apply to the $k$-fold incentive process for all $k$.

\begin{proposition}
The stationary distribution of a Markov process $X$ is the same as the stationary distribution of the process $X^k$, defined by the $k$-th power of the transition matrix.
\end{proposition}
\begin{proof}
The stationary distribution of a Markov chain can be obtained by the rows of the matrix defined by
\[ s = \lim_{m \to \infty}{T^m} = \lim_{m \to \infty}{\left(T^k\right)^m}, \]
and so the stationary distributions of the $k$-fold incentive process are the same for all $k$, given a fixed transition matrix $T$.
\end{proof}

The next proposition follows easily from the fact that the $k$-fold incentive process has the same stationary distribution and the definition of matrix multiplication.
\begin{proposition}
If a Markov process $X$ satisfies the detailed-balance condition, so does the process $X^k$ defined by the $k$-th power of the transition matrix.
\end{proposition}

\begin{corollary}
Theorem \ref{main_thm_local} holds for the $k$-fold incentive process for two-type populations.
\end{corollary}

\subsubsection*{Examples}
As an illustration of how large $N$ may need to be, we consider the $q$-replicator incentive and game matrix defined by $a=1=d$, $b=2=c$. Direct substitution shows that $\bar{a}_1 = \frac{1}{2}$ is a solution to the ISS candidate equation for all $q$ and $\mu$.  (Note that for $q=1$, $a=0$ and $a=N$ are also solutions.) For the two stationary extrema equations, we have that $\bar{a}_1 = \frac{1}{2} \pm \frac{1}{2N}$. For large $N$, these solutions all converge to $\bar{a} = (\frac{1}{2}, \frac{1}{2})$. Even for small $N$, these solutions differ only by a small amount ($\approx \frac{1}{2N}$). Similarly, for the landscape given by $a=0=d$, $b=1=c$, we can give an explicit formula for the stationary distribution for arbitrary $\mu$ (see \cite{harper2013inherent}):
\begin{align}
 s_{(0,N)} = s_{(N,0)} &= \frac{1}{2 + 2 \mu (2^N - 2)} \notag\\
 s_{(j, N-j)} &= \frac{2 \mu}{2 + 2 \mu (2^N - 2)} \binom{N}{j}
 \label{stationary_example}
\end{align}

Since only the binomial factor depends on $j$, it is clear that $(N/2, N/2)$ is the local maximum of the stationary distribution for any $\mu$ and any $N \geq 2$. So we see explicitly that the property of being stationary stable is dependent on the population size $N$.

Let the fitness landscape be defined by the $q$-replicator incentive and the game matrix $a=2=b$, $c=1=d$ (the Moran landscape with relative fitness $r=2$) and let $\mu = 0.001$ for $N=100$. As shown in \cite{harper2013incentive}, for $q \neq 1$ there is an internal candidate ISS. We consider three cases. For $q=1$, there is no ISS (the transition probabilities never intersect), and no internal local maximum of the stationary distribution. The latter is true for $q=2$ as well, but there is a state where the transitions intersect, and $D_d(x)$ is locally minimal at this state ($x \approx (33, 67)$), which shows that a candidate ISS need not be stable in the sense of the stationary distribution. An imbalanced mutation rate can force an internal equilibrium for $q=2$, for instance with $\mu_{12} = 0.05$ and $\mu_{21} = 0.0005$. For $q=0$, the projection incentive, there is a local maximum of the stationary distribution (at $i=67$). It is unique and $D_d(x)$ is minimal.

It is possible to have multiple interior stationary stable states for the replicator incentive with $q \neq 1$. For example, the following parameters give a process with two local maxima of the stationary distribution: $N=50$, $\mu = 0.1$, $q=1.5$, and the neutral fitness defined by $a=b=c=d=1$. It is also possible to have an internal local stationary maximum as well as a boundary maximum, e.g. for $N=50$, $\mu_{12} = 0.1$, $\mu_{21} = 0.01$, $q=0.5$, fitness landscape defined by $a=20$, $b=1$, $c=7$, $d=10$. 

\section{Main theorem for Globally Balanced Incentive Processes}

In general, incentive processes for more than two-types need not be locally-balanced, so we cannot use Equation \ref{s_j} to establish the main theorem. In fact, it appears to rarely be the case (the neutral landscape is an exception in some cases), though computations indicate that in some cases, the detailed balance condition is nearly satisfied, and Equation \ref{s_j} produces a good approximation of the stationary distribution. Regardless, we now generalize Theorem \ref{main_thm_local} to globally-balanced incentive processes. The proof strategy is similar to the locally-balanced case, using the global balance equation instead. 

Any probability inflow-outflow balanced state, i.e. any state such that $\sum_{a \text{ adj } b}{T_{a}^{b}} = \sum_{b \text{ adj } a} {T_{b}^{a}}$, will minimize the distance between the expected next state and the current state. We will call such states \emph{probability flow neutral}. This could happen in principle at e.g. a saddle point, but since we are mostly concerned with stationary maxima for stability purposes, we will focus on maxima.

\begin{proposition}
For a continuous incentive and mutation matrix and sufficiently large $N$, local maxima and minima of the stationary distribution are probability flow neutral. Precisely, for $\epsilon > 0$, there is an $N'$ such that for $N > N'$,
\[ \left| \sum_{a \text{ adj } b}{T_{a}^{b}} - \sum_{b \text{ adj } a}{T_{b}^{a}} \right| < \epsilon \]
\end{proposition}

\begin{proof}
We prove just the maximal case. Suppose we have local maxima $a$ of the stationary distribution and let $\epsilon > 0$. Since all the transition probabilities are continuous in the population state, so is the stationary distribution, which is therefore uniformly continuous and bounded by $M$.
For sufficiently large $N$ we have that $|s_b/s_a - 1| < \epsilon/ (n^2 M)$ for states $b$ adjacent to $a$. By the local balance equation,
\[s_a \sum_{a \text{ adj } b}{T_{a}^{b}} = \sum_{b \text{ adj } a} {s_b T_{b}^{a}},\]
so dividing through by $s_a$ and using the above inequality for $s_b/s_a$,
\[ \left| \sum_{a \text{ adj } b}{T_{a}^{b}} - \sum_{b \text{ adj } a}{T_{b}^{a}} \right| = \left|\sum_{b \text{ adj } a} \left(\frac{s_b}{s_a}-1\right)T_b^{a} \right| < \epsilon / (n^2 M) \left|\sum_{b \text{ adj } a}{T_b^{a}} \right| < \epsilon.\]

\end{proof}

We take it as given that stationary extrema occur at probability flow neutral states, knowing that the population states may differ by a small amount. Now we show that stationary extrema minimize the distance functions $D_d$. While the relative entropy $(d=1)$ is only well-defined on the interior of the simplex, the same result holds when the state space is restricted to a boundary simplex of lower dimension. This means that the relative entropy can still detect boundary stable points of the same process on the restricted subspace, but the other $D_d$ do so on the full simplex.

\begin{theorem}
For a continuous incentive and mutation matrix on $n$ types, there is a sufficiently large $N$ so that probability flow neutral states of the incentive process are local minima of the distance functions $D_d$ for $0 \leq d < 1$; interior extrema are local minima of the relative entropy $D_1$.
\label{main_thm_global}
\end{theorem}
\begin{proof}
Since the proof is similar to the two-type case, we just sketch the proof. Let $\epsilon > 0$ and suppose that $a$ is a probability flow neutral state of the process:
\[ \left| \sum_{a \text{ adj } b}{T_{a}^{b}} - \sum_{b \text{ adj } a} {T_{b}^{a}} \right| < \epsilon . \]
We have that $E(\bar{a}) - \bar{a} = \frac{1}{N}\sum_{\alpha \neq \beta}{ \left( T_{a}^{a + i_{\alpha \beta}} - T_{a}^{a - i_{\alpha \beta} }\right) i_{\alpha \beta} }$, and we use continuity to choose a sufficiently large $N$ so to replace right-hand terms with $T_{a + i_{\alpha \beta}}^{a}$ (as in the lemma before Theorem 1). Then we have that
\[D_d(a) < \left| \sum_{a \text{ adj } b}{T_{a}^{b}} - \sum_{b \text{ adj } a} {T_{b}^{a}}\right|^2 < \frac{n}{N} \epsilon^2.\]
\end{proof}

Since the $k$-fold incentive process has the same stationary distribution as the $k=1$ case, and since the expected value functions agree at equilibrium via Proposition \ref{expected_incentive}, we have the next theorem as an immediate corollary. This shows in particular that the connection between the stationary distribution and evolutionarily stable states can hold for generational processes.

\begin{theorem}
Theorem \ref{main_thm_global} holds for the $k$-fold incentive process.
\end{theorem}

\subsection{Three Type Examples}

Here we give several interesting examples using I.M. Bomze's classification of three-type phase portraits for the replicator dynamic with linear fitness landscape (see \cite{bomze1983lotka} and the additions and corrections \cite{bomze1995lotka}). We use the Fermi incentive for all examples. A full list corresponding to each of the 48 phase portraits is available online at \url{http://people.mbi.ucla.edu/marcharper/stationary_stable/3x3/}. Figure \ref{three_player_plots} gives stationary distributions and expected distances for the following game matrices, to illustrate Theorem \ref{main_thm_global}.

\[ \left(\begin{smallmatrix}
0 & 0 & 1\\ 
0 & 0 & 1\\
1 & 0 & 0\\
\end{smallmatrix} \right) \quad
\left(\begin{smallmatrix}
0 & 0 & 0\\ 
0 & 0 & -1\\
0 & -1 & 0\\
\end{smallmatrix}\right)  \quad
\left(\begin{smallmatrix}
0 & 2 & 0\\ 
2 & 0 & 0\\
1 & 1 & 0\\
\end{smallmatrix}\right)
\]

%

\begin{landscape}

\centering
\begin{figure}[h]
        \begin{subfigure}[b]{0.4\textwidth}
            \centering
            \includegraphics[width=\textwidth]{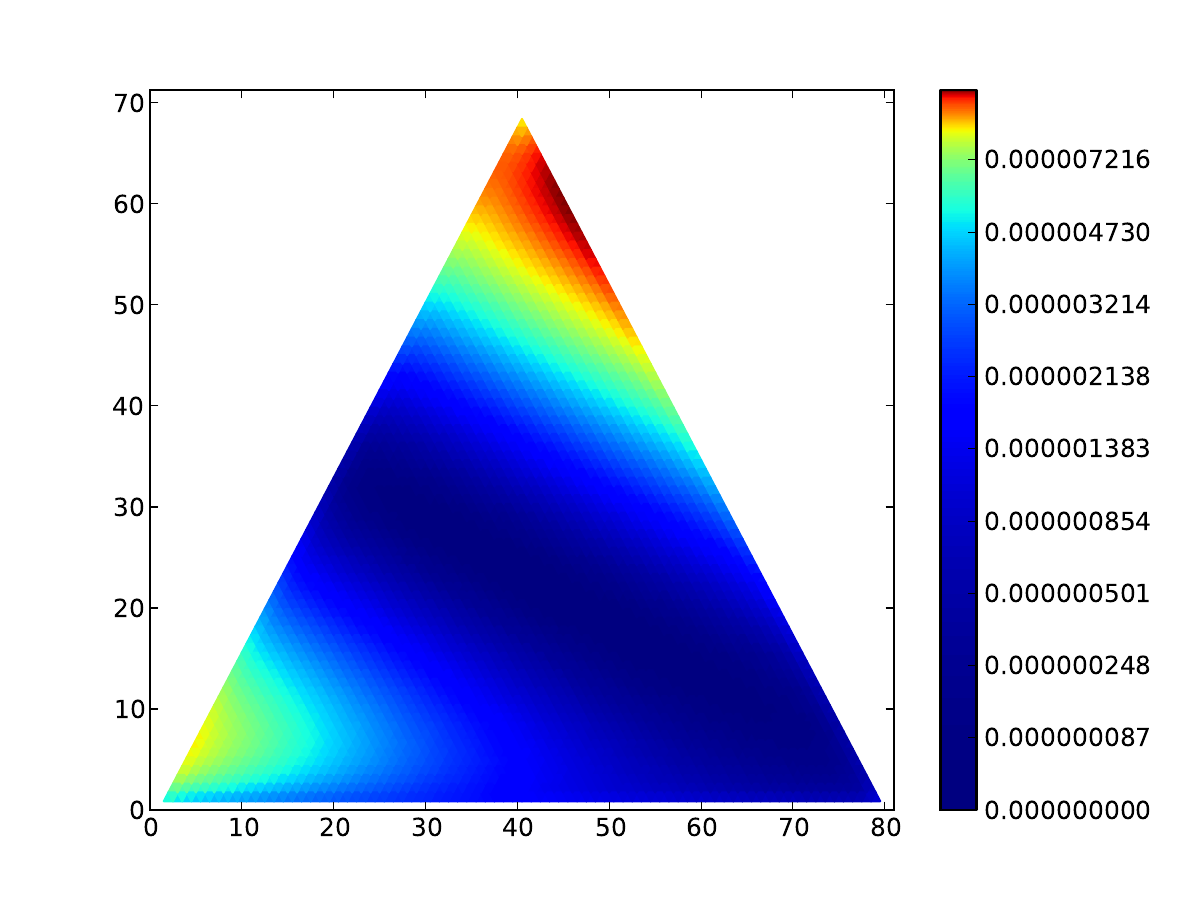}
        \end{subfigure}%
        ~ 
        \begin{subfigure}[b]{0.4\textwidth}
            \centering
            \includegraphics[width=\textwidth]{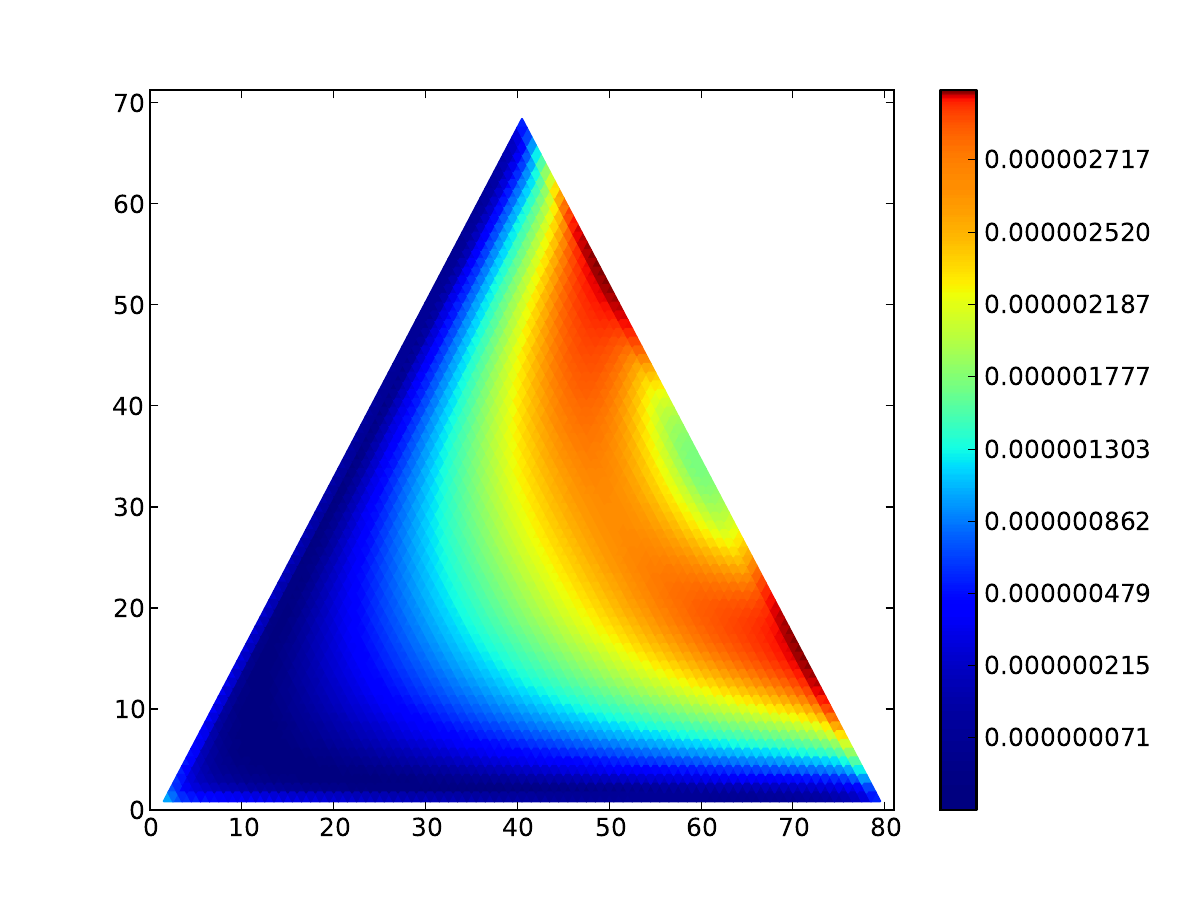}
        \end{subfigure}
        \begin{subfigure}[b]{0.4\textwidth}
            \centering
            \includegraphics[width=\textwidth]{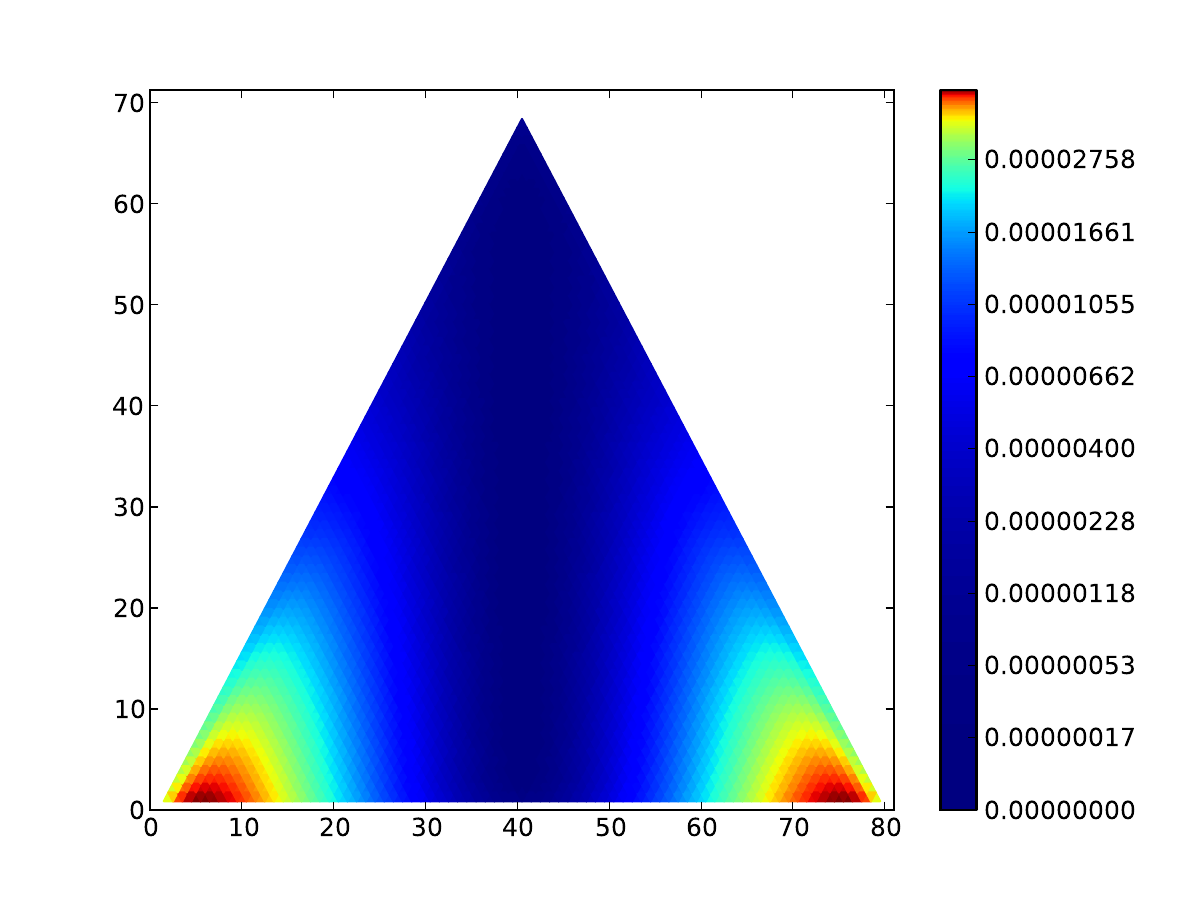}
        \end{subfigure}
        \\
        \begin{subfigure}[b]{0.4\textwidth}
            \centering
            \includegraphics[width=\textwidth]{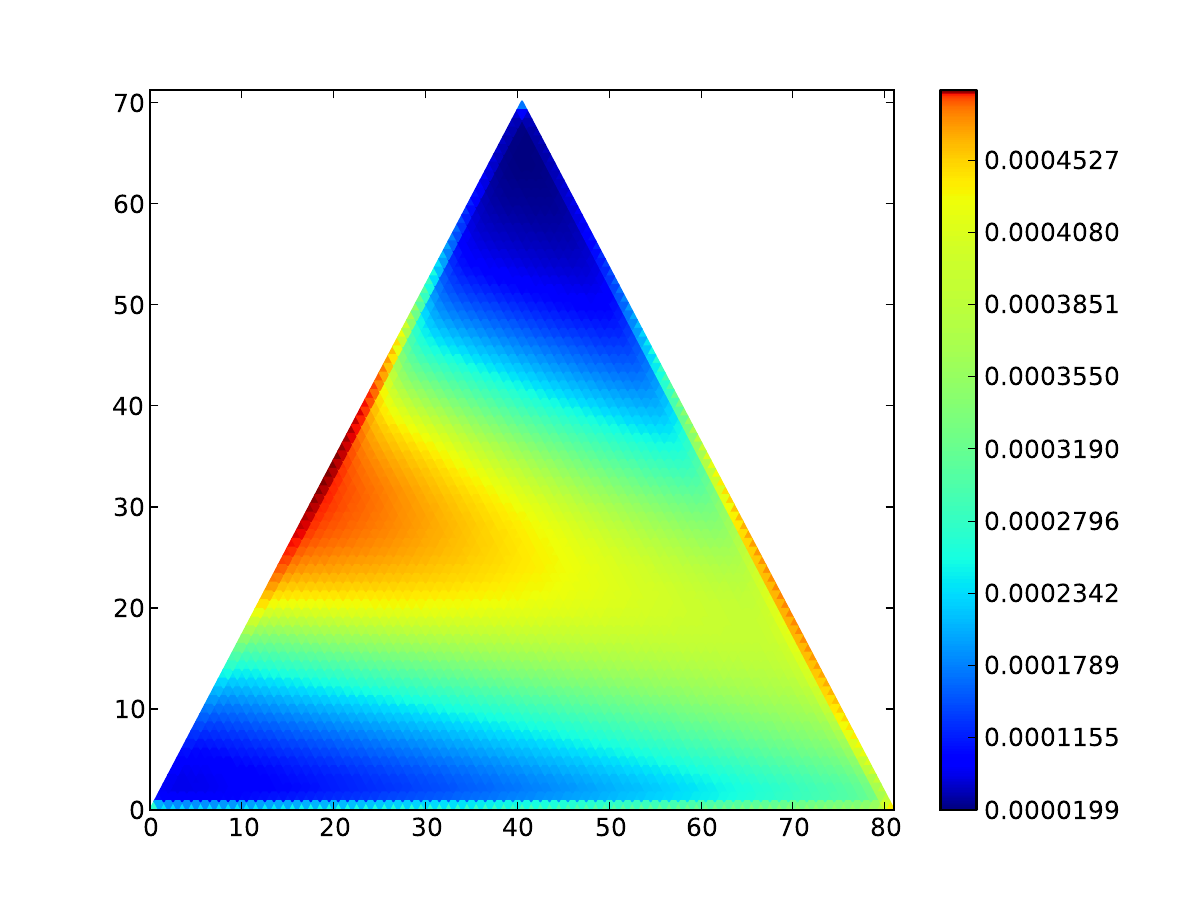}
        \end{subfigure}%
        ~ 
        \begin{subfigure}[b]{0.4\textwidth}
            \centering
            \includegraphics[width=\textwidth]{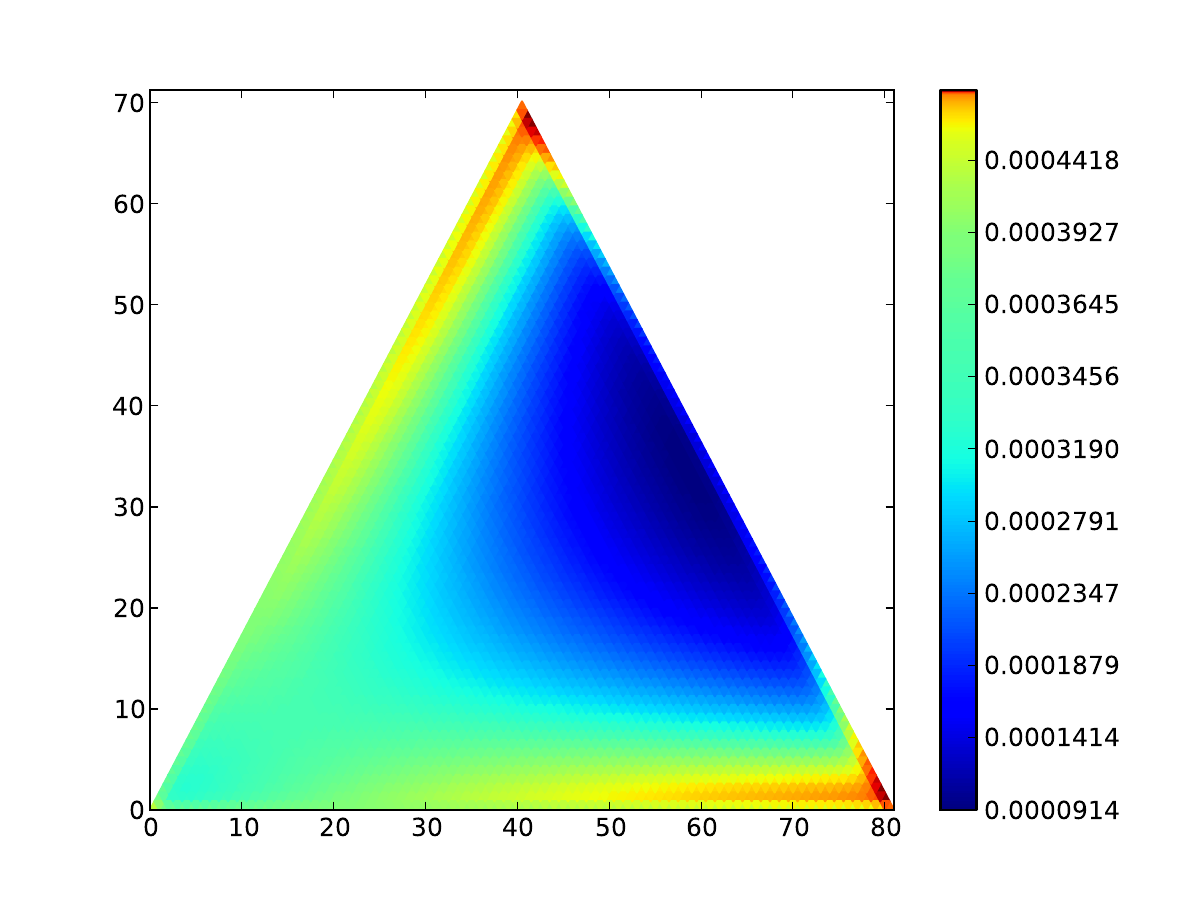}
        \end{subfigure}
        \begin{subfigure}[b]{0.4\textwidth}
            \centering
            \includegraphics[width=\textwidth]{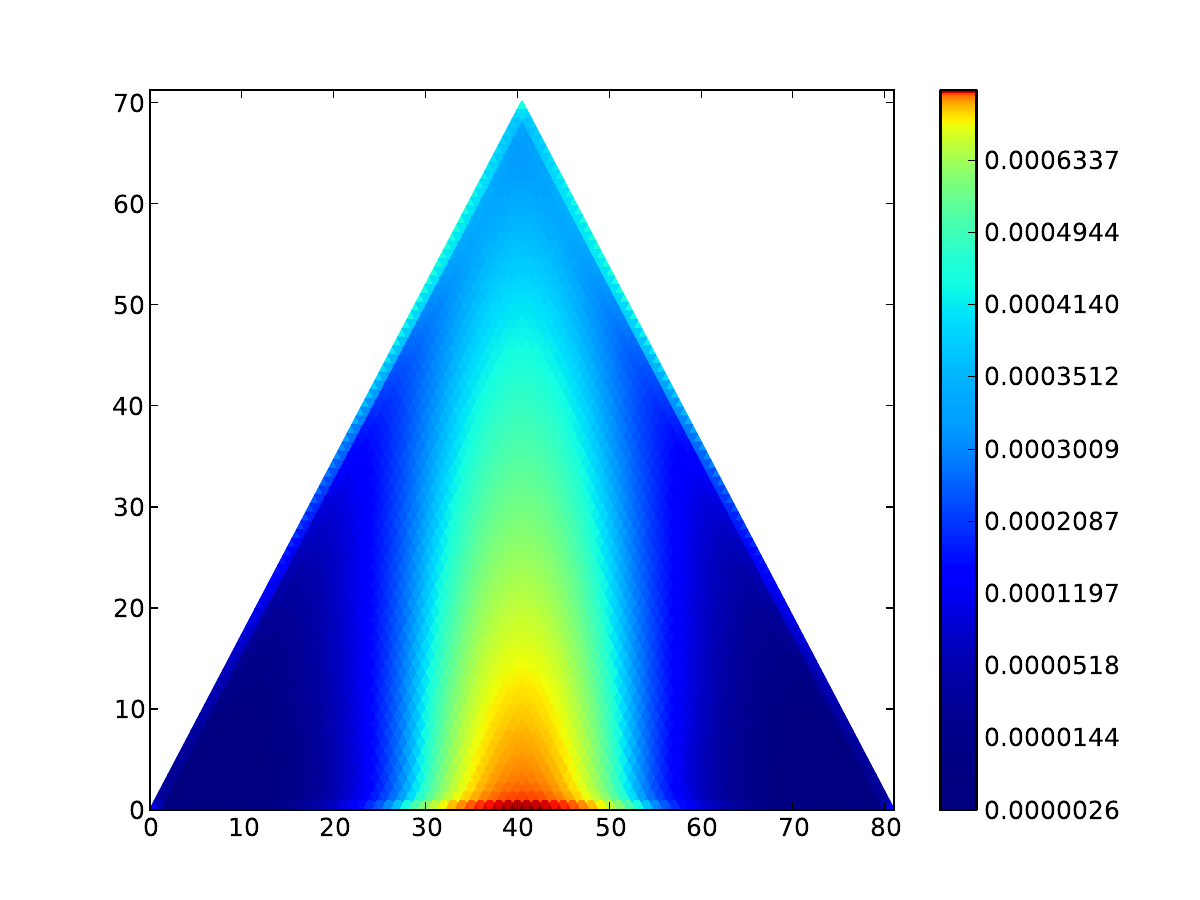}
        \end{subfigure}
        \\
        \caption{Top Row: Relative entropy of the expected next state with the current state $D_1(a)$ for Fermi incentives ($q=1$, $\beta=1$) for the incentive process with $\mu=\frac{3}{2}\frac{1}{N}$ for game matrices 2, 20, and 47 in Bomze's classification. Bottom Row: Stationary distributions for the incentive process for the same parameters.}
        \label{three_player_plots}
\end{figure} 
\end{landscape}

\begin{landscape}

\subsection{Rock-Scissors-Paper Games}

The rock-scissors-paper game is given by a matrix of the form:
\small{\[ \left( \begin{smallmatrix}
0 &-1 & 1 \\
1 & 0 &-1 \\
-1& 1 & 0
\end{smallmatrix}\right) \]}

In \cite{andrae2010entropy} (online supplement), Andrae et al give steady-state distributions for the rock-scissors-paper game in the context of entropy production. Figure \ref{rsp_plots} contains stationary distributions for the rock-scissors-paper game; these are qualitatively similar to those in Figure 4 of \cite{andrae2010entropy}. Traulsen et al consider the average drift of a relative entropy equivalent Lyapunov quantity in \cite{traulsen2009stochastic} for the Moran process in a finite population, and find that convergence depends on the value of $N$. The rock-scissor-paper game for the Moran process does not yield a detail-balanced Markov process \cite{andrae2010entropy}, and because of the cyclic-nature of the process, the transitions are particularly removed from detailed-balance. This is reflected in the fact that the stationary distribution can take many iterations of the transition matrix to converge. These plots also illustrate that the stability of particular states depends significantly on the mutation rate.


\begin{figure}[h]
        \begin{subfigure}[b]{0.4\textwidth}
            \centering
            \includegraphics[width=\textwidth]{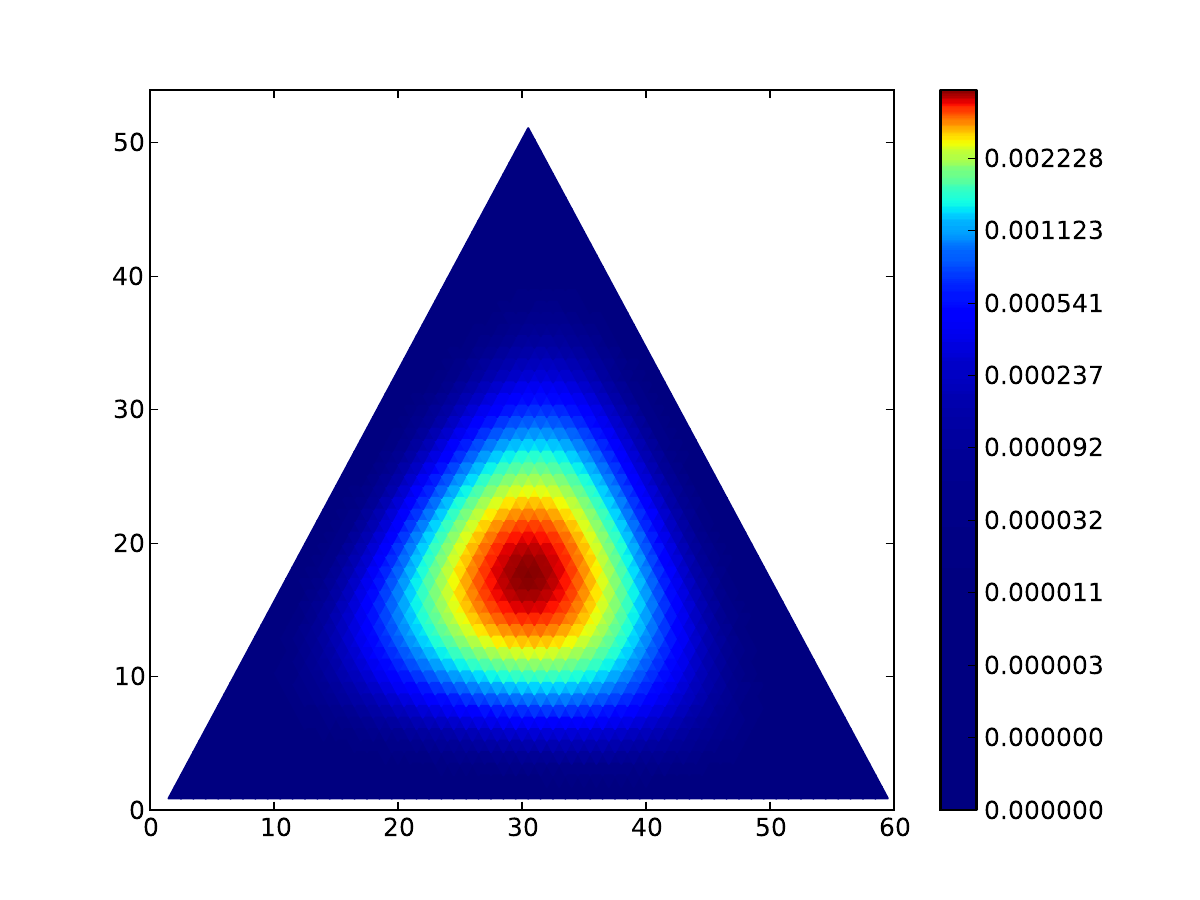}
        \end{subfigure}%
        ~ 
        \begin{subfigure}[b]{0.4\textwidth}
            \centering
            \includegraphics[width=\textwidth]{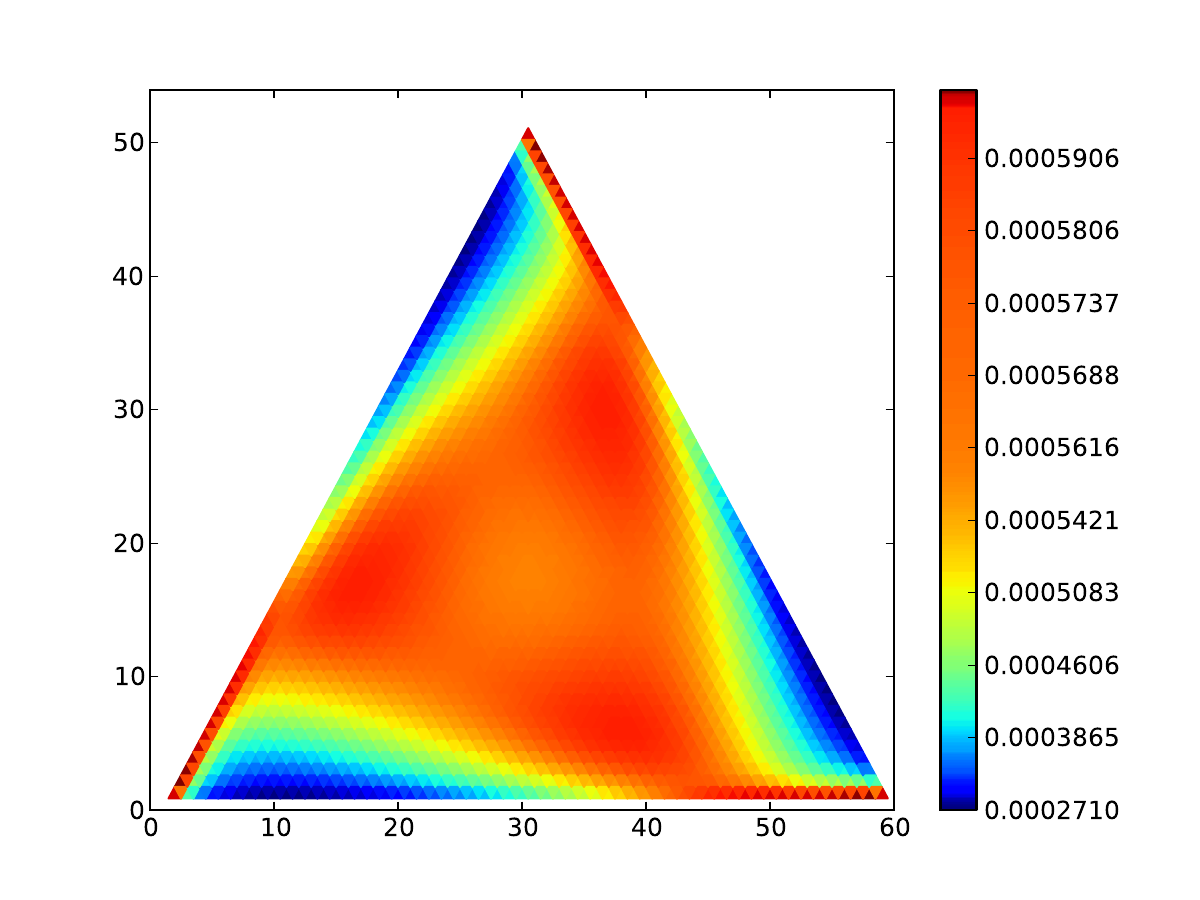}
        \end{subfigure}
        \begin{subfigure}[b]{0.4\textwidth}
            \centering
            \includegraphics[width=\textwidth]{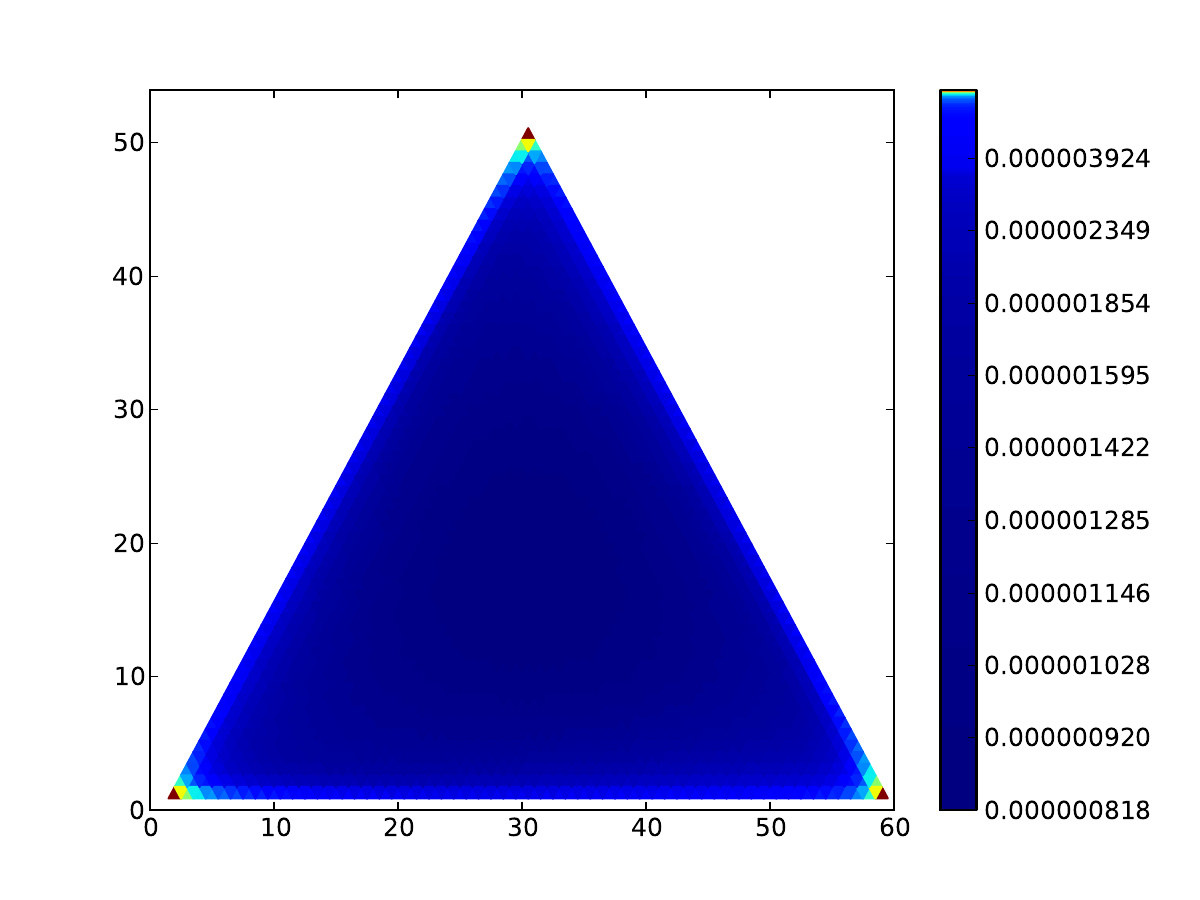}
        \end{subfigure}
        \caption{Stationary distributions for the rock-scissors-paper game for $N=60$ with the Fermi incentive, $\beta=1$. Left to right we have $\frac{2}{3}\mu = \frac{1}{\sqrt{N}}, \frac{1}{N}, \frac{1}{N^{3/2}}$.}
        \label{rsp_plots}
\end{figure} 
\end{landscape}

\subsection{Four Types}

For completeness, we include a higher dimensional example. Consider the incentive process defined by the Fermi incentive, $N=60$, $\mu=1/N$ and the game matrix 
\small{\[ M_4 = \left( \begin{smallmatrix}
0 & 1 & 1 & 1\\
1 & 0 & 1 & 1\\
1 & 1 & 0 & 1\\
0 & 0 & 0 & 1\\
\end{smallmatrix}\right) \]}

Although we cannot easily plot the full stationary distribution, we can plot the two-dimensional boundary simplices. Figure \ref{figure_4_d} gives the $D_0$ expectation distance and the stationary distribution for two of the four faces. Three are the same; the fourth, where $a_4=0$, is distinct, and similar to the stationary distribution in Figure \ref{figure_graphical_abstract}. From inspection of the game matrix, we would predict that $(N/3, N/3, N/3)$ is an ISS for the distinct face. The three similar faces have boundary ISS such as $(N/2,0,N/2)$. Note that the stationary distribution and $D_0$ expected distances are computed for the full process on the three-dimensional state space; only the two-dimensional faces are plotted.

\begin{figure}[h]
        \begin{subfigure}[b]{0.4\textwidth}
            \centering
            \includegraphics[width=\textwidth]{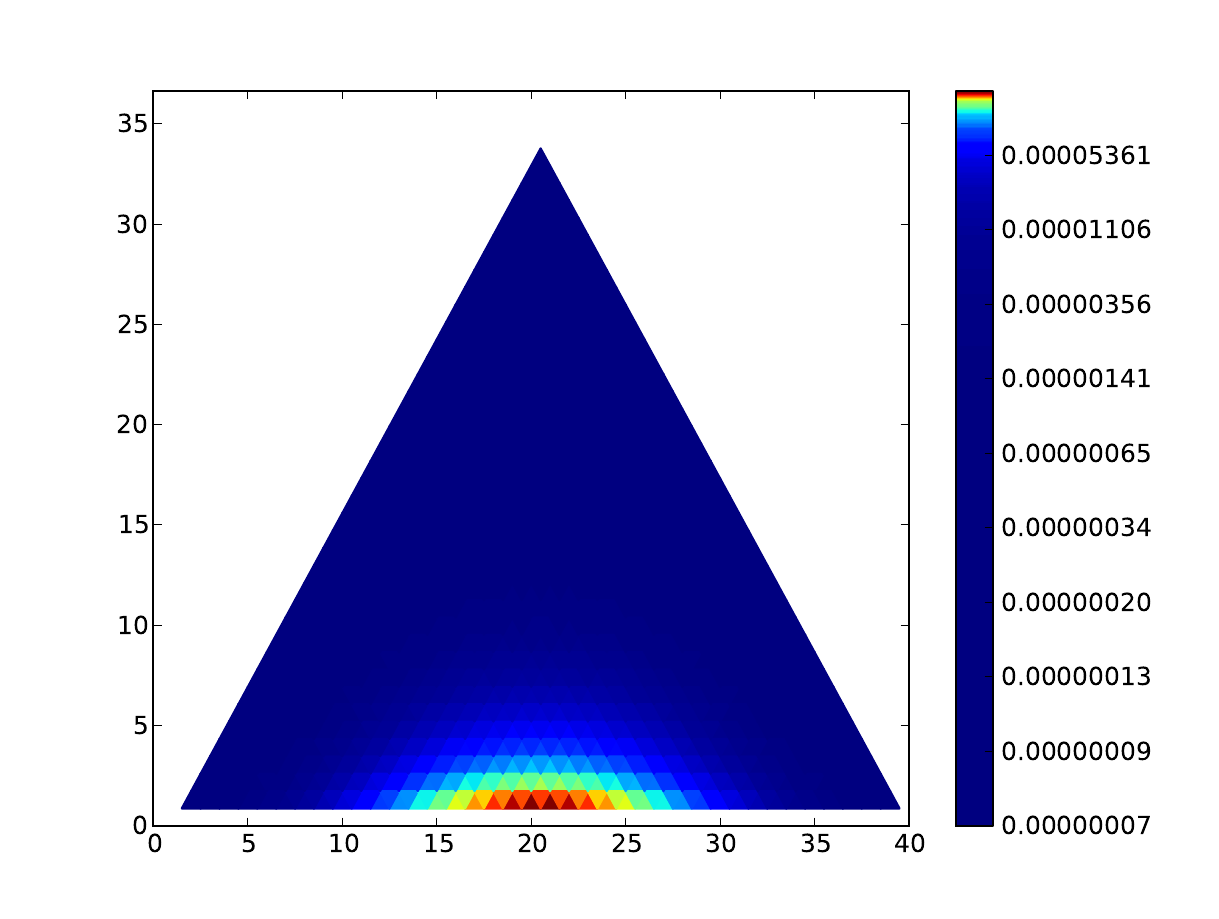}
        \end{subfigure}%
        ~ 
        \begin{subfigure}[b]{0.4\textwidth}
            \centering
            \includegraphics[width=\textwidth]{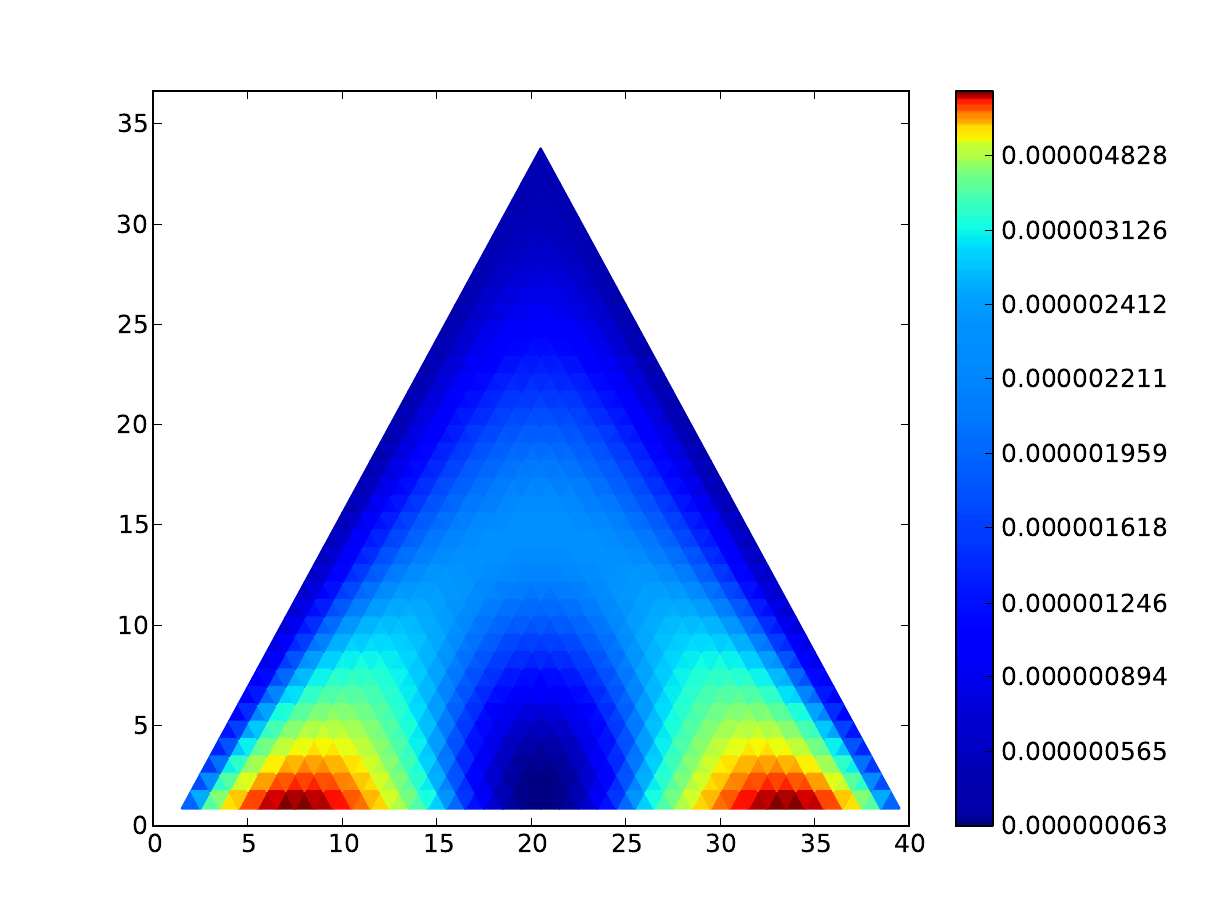}
        \end{subfigure}
        \\
        \begin{subfigure}[b]{0.4\textwidth}
            \centering
            \includegraphics[width=\textwidth]{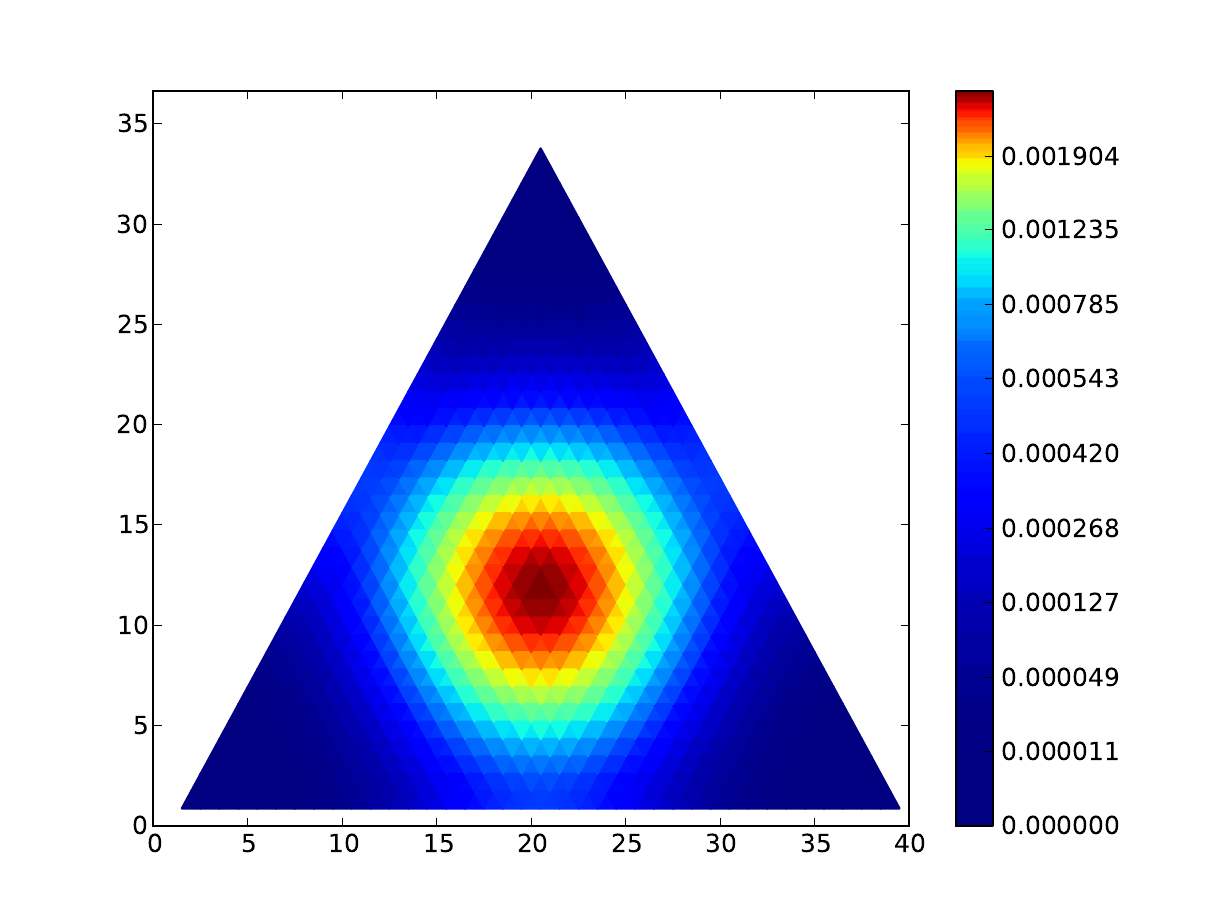}
        \end{subfigure}%
        ~ 
        \begin{subfigure}[b]{0.4\textwidth}
            \centering
            \includegraphics[width=\textwidth]{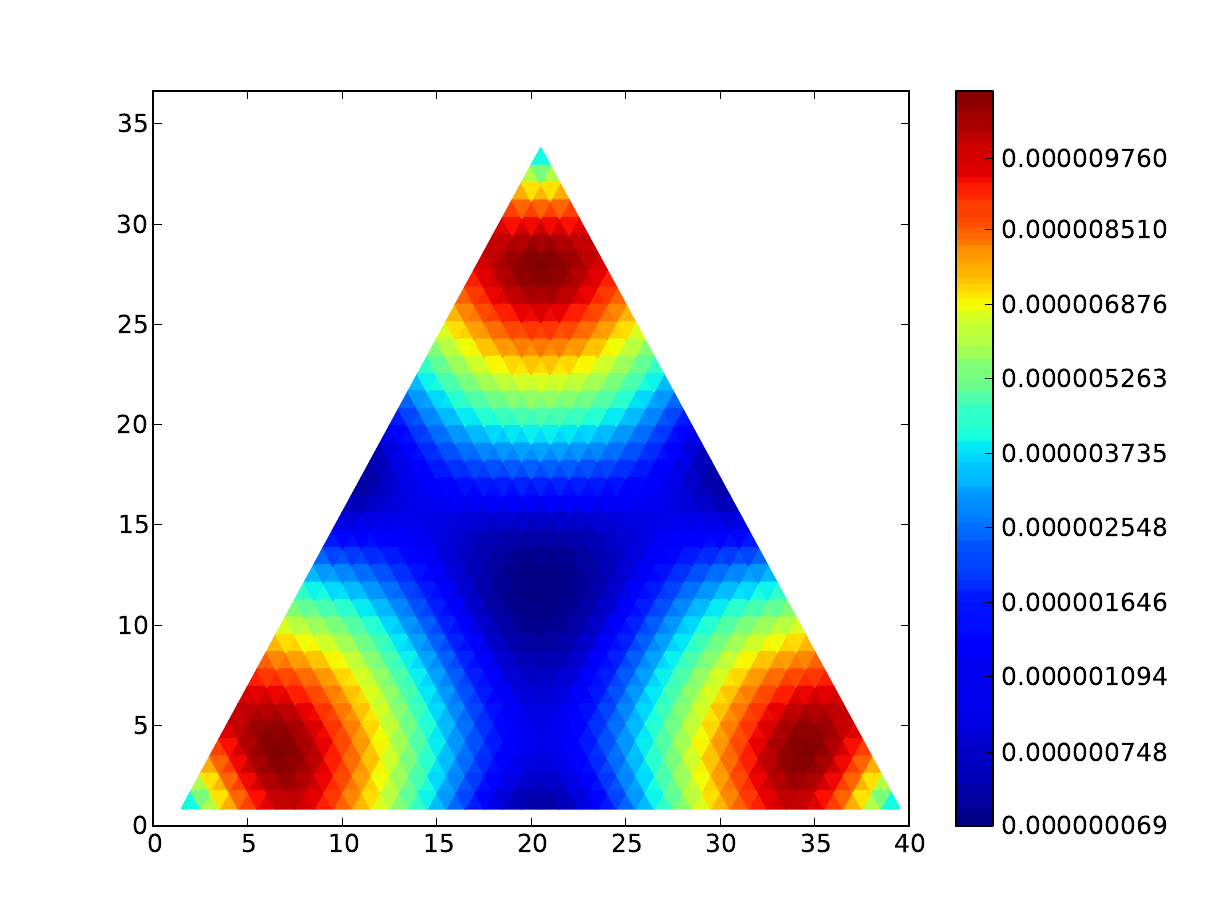}
        \end{subfigure}
        \caption{Top Left: Stationary distribution for the three common faces for the incentive process for matrix $M_4$ above, $N=40$. Bottom Left: Stationary distribution for the face with $a_4=0$. Right: $D_0$ expected distances for the faces to the left. }
        \label{figure_4_d}
\end{figure} 

\section{The Wright-Fisher Process}

The Wright-Fisher process behaves differently than the Moran process for many combinations of parameters; see, e.g. \cite{imhof2006evolutionary}. In general, even for two-types, the Wright-Fisher process is not locally balanced \cite{taylor2006symmetry}. Because of the combinatorial complexity of the process, the Wright-Fisher process is often studied via diffusion-type approximations. In some cases, these differential equations produce a qualitatively similar result to our main theorem for the incentive process. For instance in \cite{chalub2011frequency}, the authors show that for the replicator incentive, the limiting distribution of the diffusion approximation concentrates on ESS. Combined with the fact that $E(\bar{a})$ is the same as for the incentive process (Prop \ref{expected_incentive}), one can reasonably expect some subset of the parameters to produce a similar result to Theorem \ref{main_thm_local} for the Wright-Fisher process. See Figure \ref{three_player_plots} for examples.


We give an partial analog of Theorem \ref{main_thm_local} for the Wright-Fisher process. Crucially, note that a local maximum of the stationary distribution of the Wright-Fisher process is in fact a global maximum since every state is connected directly to every other state (assuming that that transition probabilities are never zero), and therefore is necessarily unique.

\begin{theorem}
Suppose a given incentive and mutation matrix are continuous. For the Wright-Fisher process, suppose that the stationary distribution has a global maximum at $a$, is symmetric about the maximum, and is otherwise vanishingly small (if the maximum is not central). Then for sufficiently large $N$, the state $a$ is an ISS candidate and a minimum of the distance $D_d$.
\label{main_thm_wf}
\end{theorem}
\begin{proof}
We only give the case for $n=2$ in detail. Since the stationary distribution satisfies $s=sT$, we have that
\[ s_{i+1} - s_{i} = \sum_{k}{s_k \left(T_{k}^{i+1} - T_{k}^{i}\right)},\]
where $s_{i+1}$ denotes an adjacent state and by continuity for some $j$ we must have that $T_{j}^{j+1} - T_{j}^{j-1} = 0$ (at the maximum since the distribution is symmetric). A bit of algebra shows that this occurs when
\[ (N-j) p_1(j) = (j+1) (1 - p_1(j)).\]
For sufficiently large $N$, solutions to this equation are the same as those of the ISS candidate equation. For $n > 2$, looking at all the immediately adjacent states leads to $p_{\alpha}(\bar{a}) \bar{a}_{\beta} = p_{\beta}(\bar{a}) \bar{a}_{\alpha}$ for large $N$, which implies by summation that $p(\bar{a}) = \bar{a}$.
\end{proof}

Computationally, we can compute the stationary distribution of the Wright-Fisher process for any combination of parameters for smaller $N \approx 100$ easily. Figures \ref{three_player_plots_wf} and \ref{three_player_plots_wf} give a variety of computational examples, some of which suggest that a more general version of Theorem \ref{main_thm_wf} holds. Note the sensitivity to the mutation rate $\mu$. Qualitatively, Figure \ref{figure_graphical_abstract} is similar for the Wright-Fisher process on the interior. Figures for all 48 of the landscape in Bomze's classification are also available online at \url{http://people.mbi.ucla.edu/marcharper/stationary_stable/3x3/}.

\begin{landscape}

\centering
\begin{figure}[h]
        \begin{subfigure}[b]{0.4\textwidth}
            \centering
            \includegraphics[width=\textwidth]{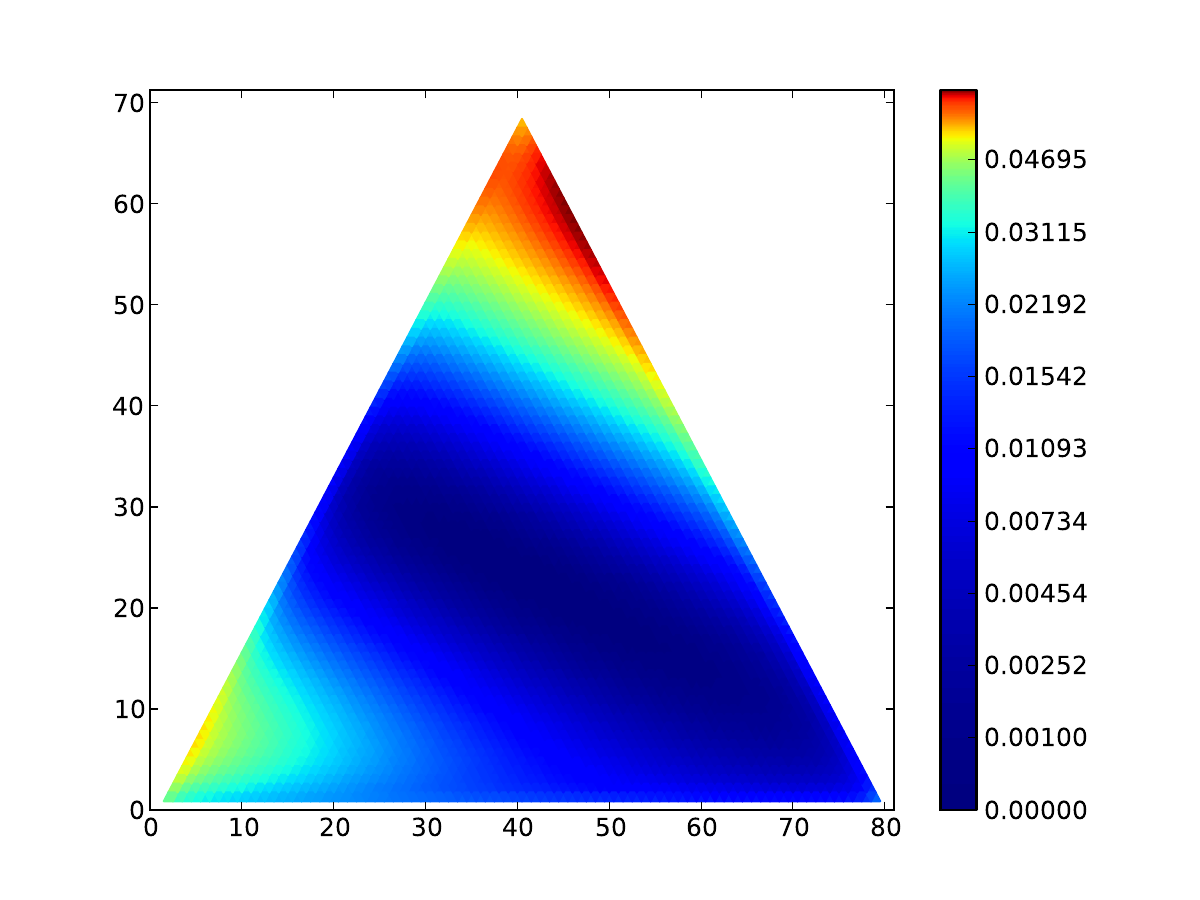}
        \end{subfigure}%
        ~ 
        \begin{subfigure}[b]{0.4\textwidth}
            \centering
            \includegraphics[width=\textwidth]{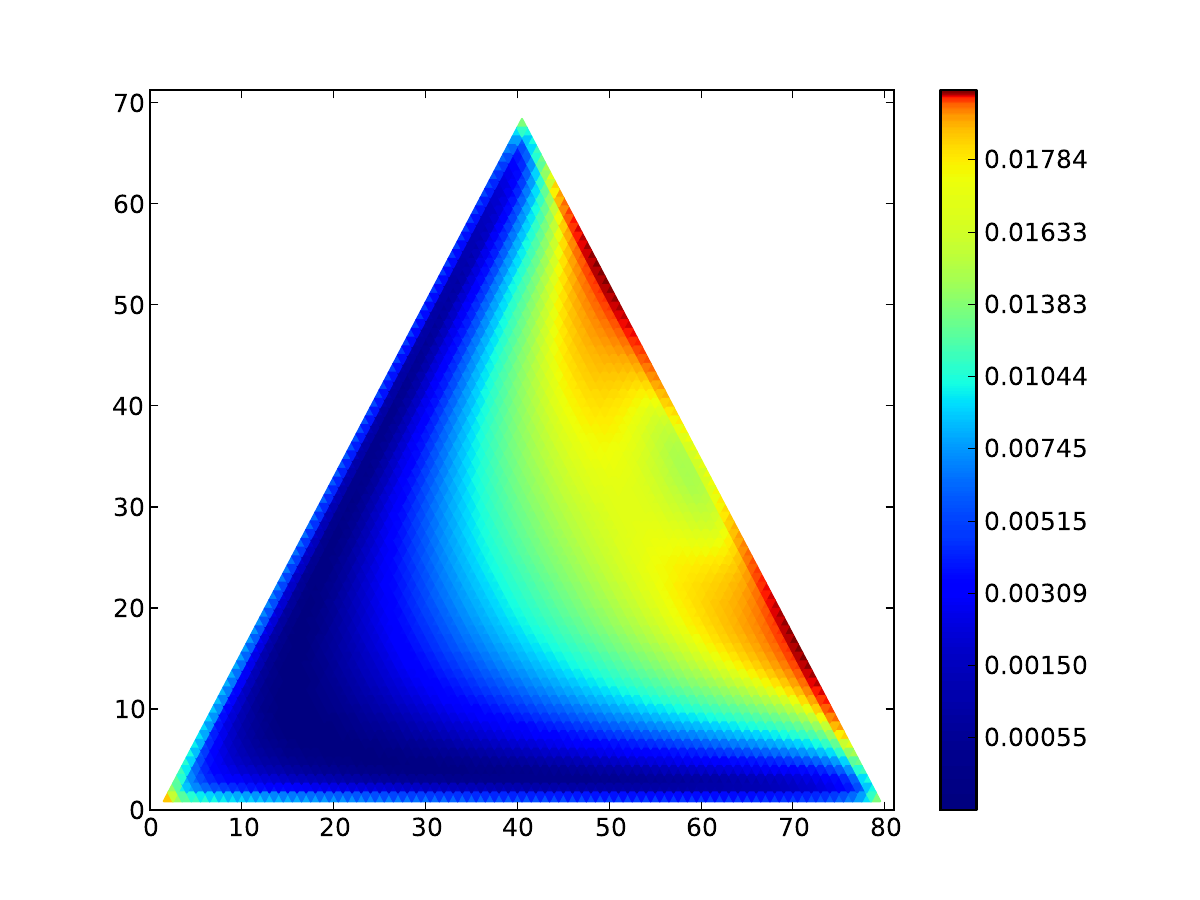}
        \end{subfigure}
        \begin{subfigure}[b]{0.4\textwidth}
            \centering
            \includegraphics[width=\textwidth]{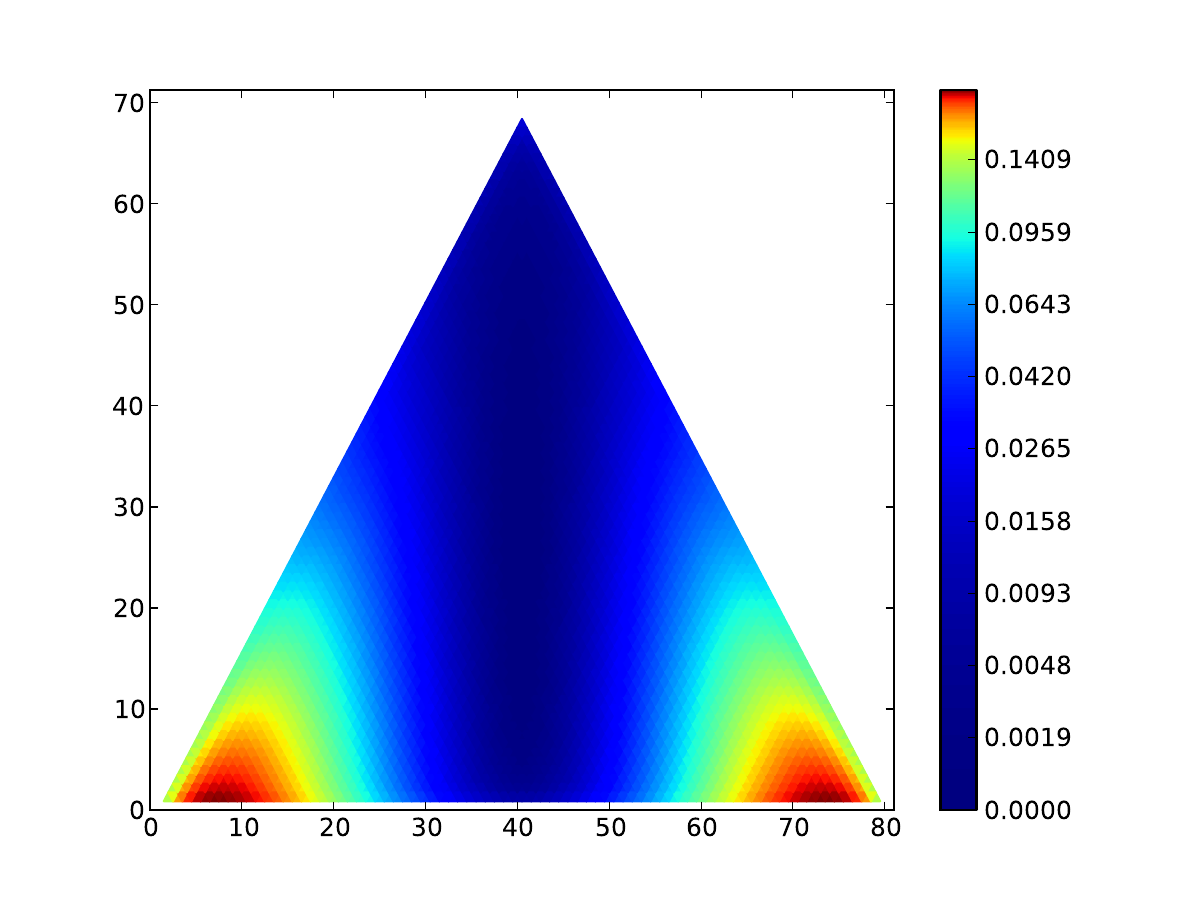}
        \end{subfigure}
        \\
        \begin{subfigure}[b]{0.4\textwidth}
            \centering
            \includegraphics[width=\textwidth]{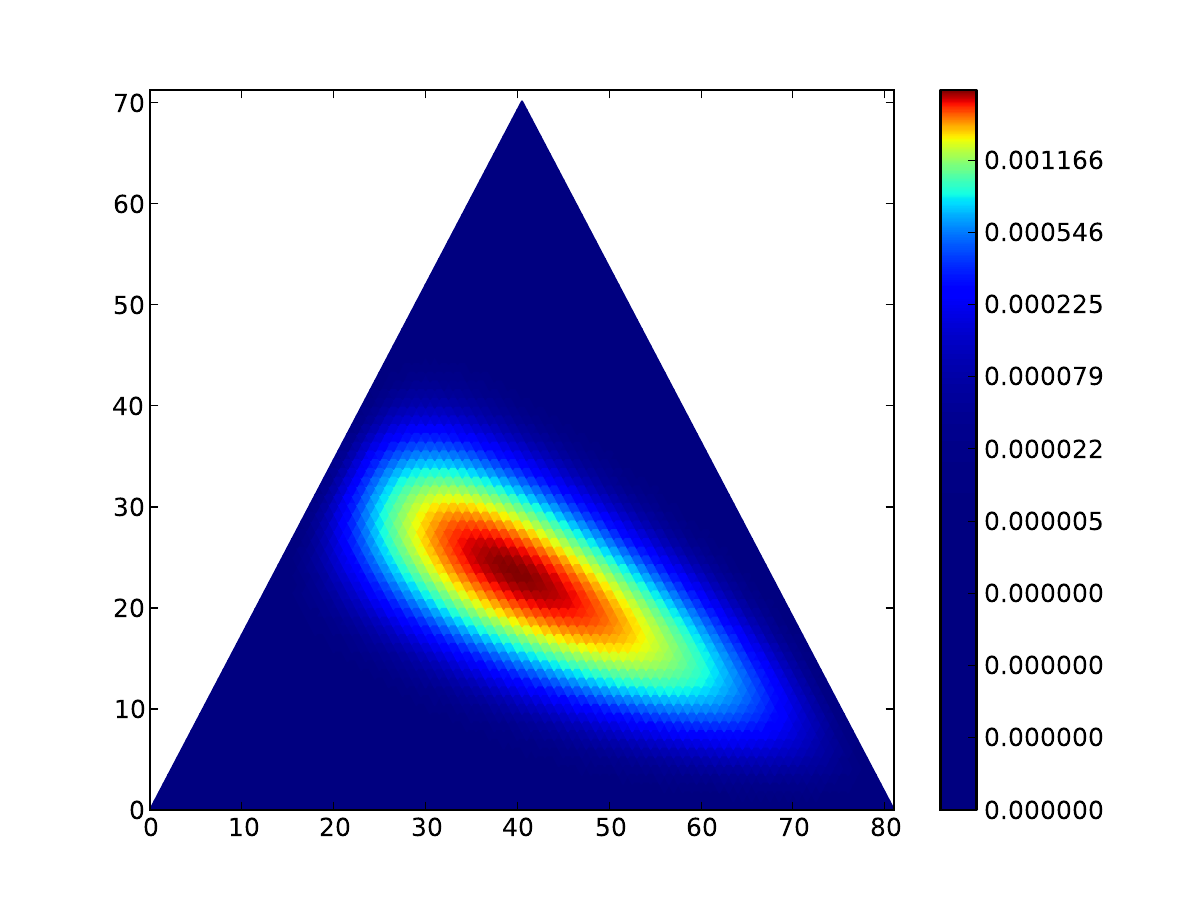}
        \end{subfigure}%
        ~ 
        \begin{subfigure}[b]{0.4\textwidth}
            \centering
            \includegraphics[width=\textwidth]{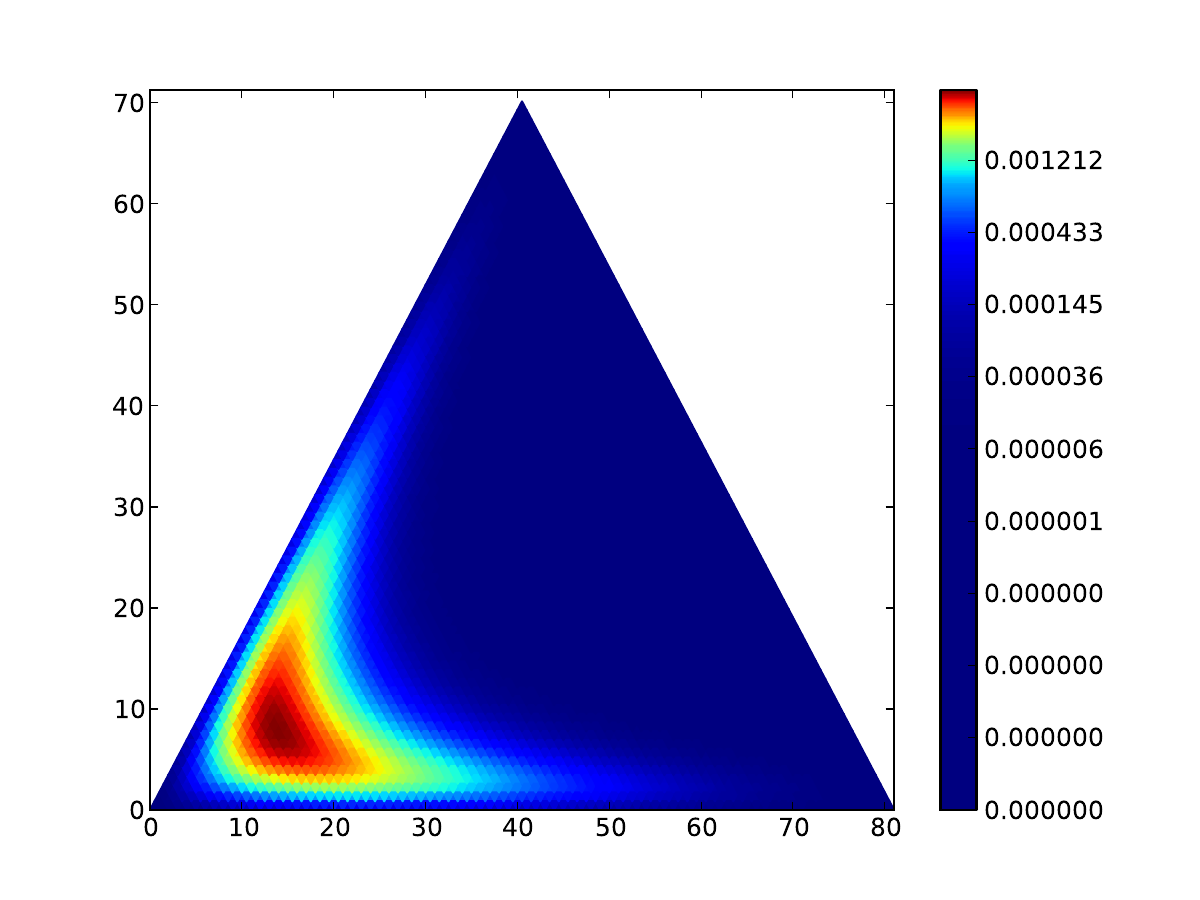}
        \end{subfigure}
        \begin{subfigure}[b]{0.4\textwidth}
            \centering
            \includegraphics[width=\textwidth]{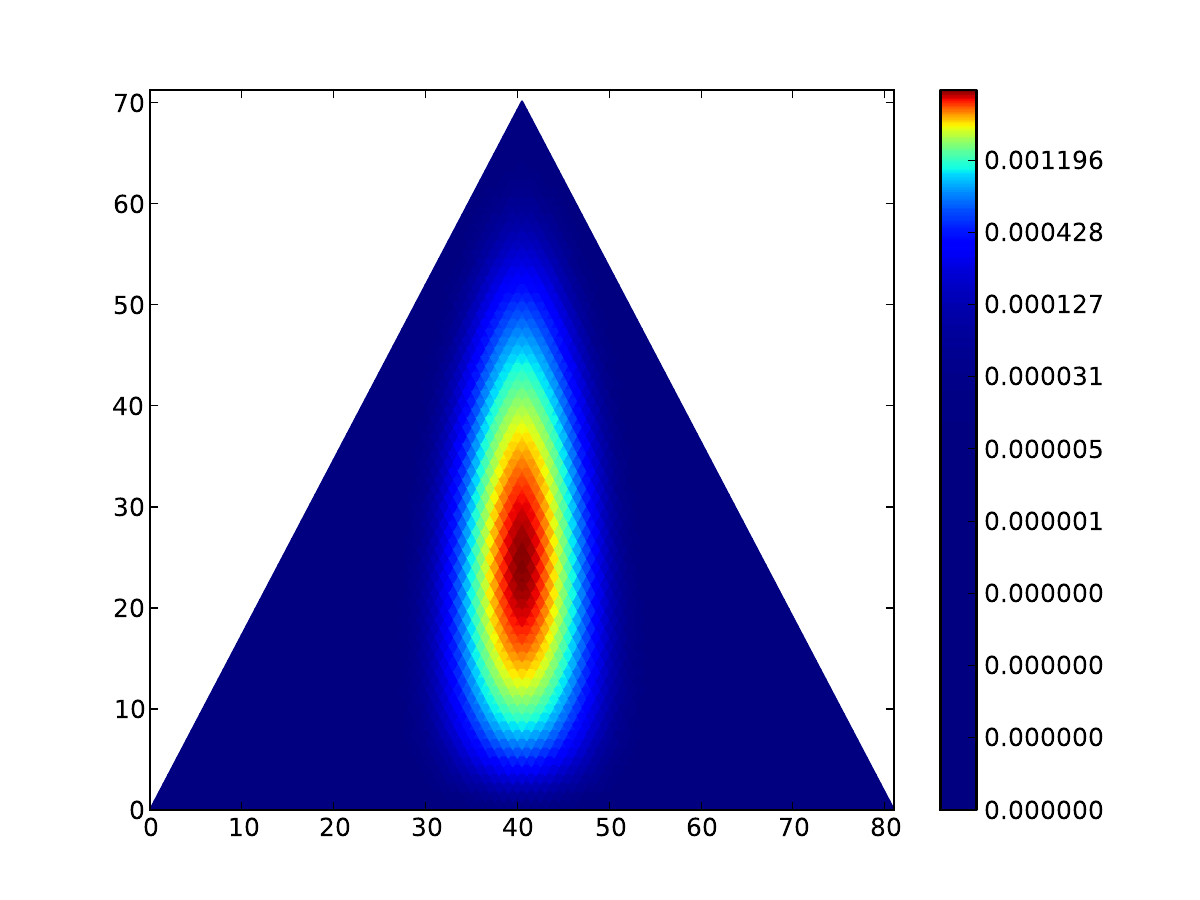}
        \end{subfigure}
        \\
        \begin{subfigure}[b]{0.4\textwidth}
            \centering
            \includegraphics[width=\textwidth]{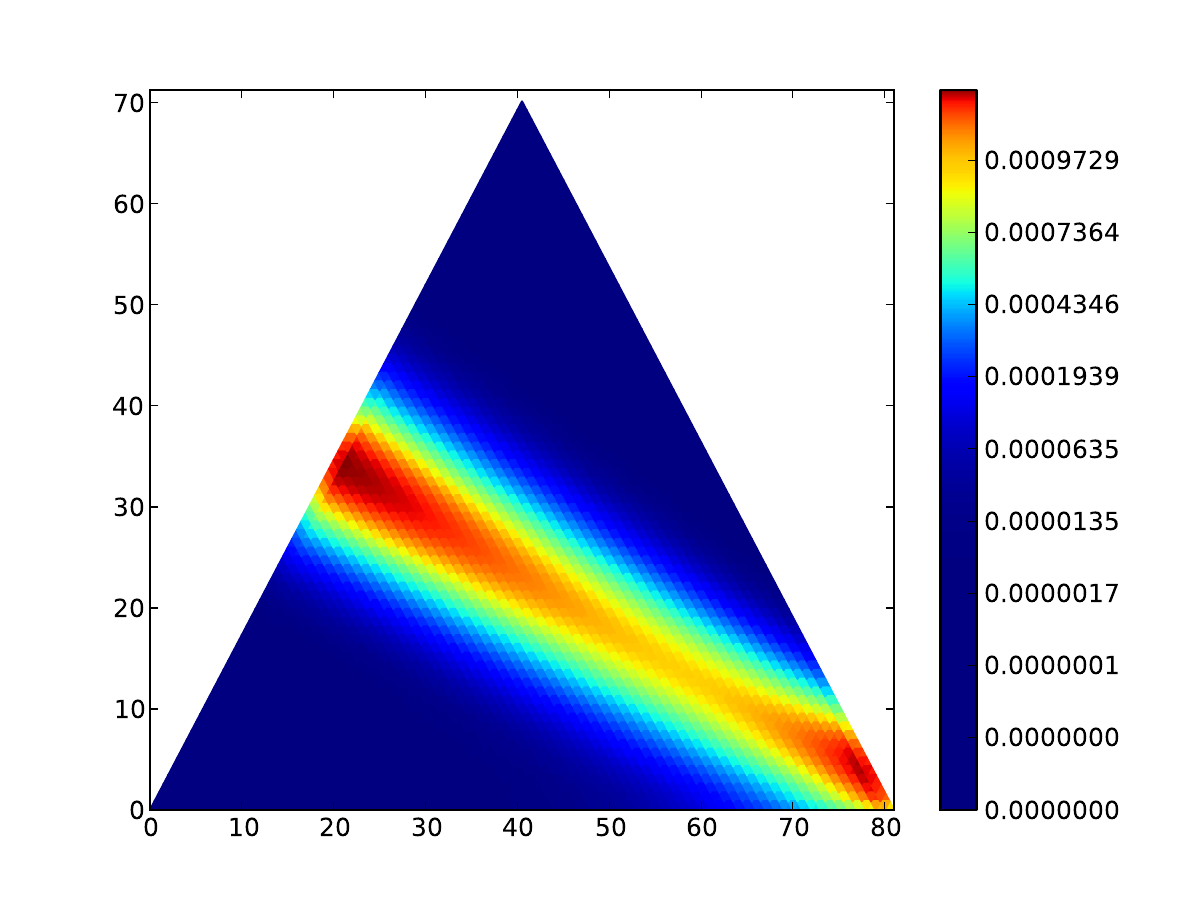}
        \end{subfigure}%
        ~ 
        \begin{subfigure}[b]{0.4\textwidth}
            \centering
            \includegraphics[width=\textwidth]{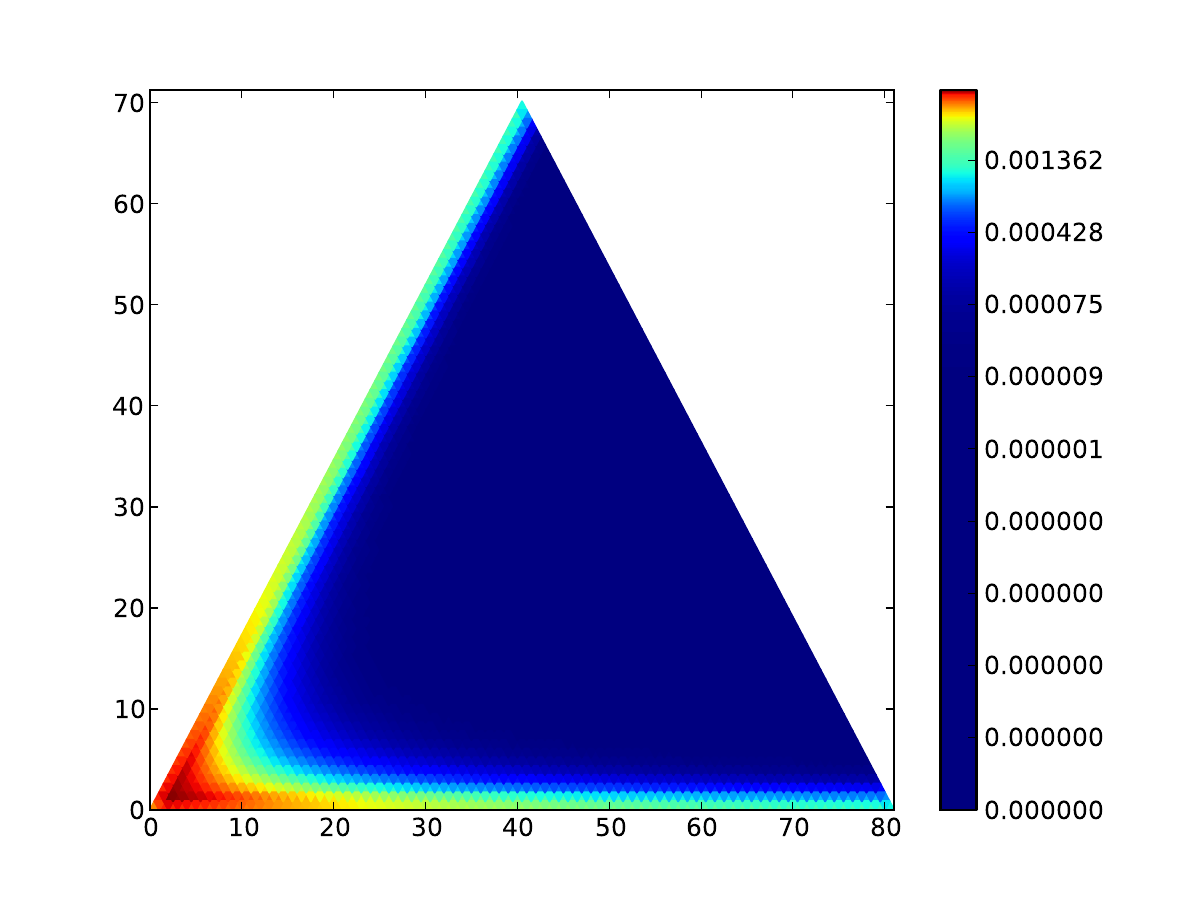}
        \end{subfigure}
        \begin{subfigure}[b]{0.4\textwidth}
            \centering
            \includegraphics[width=\textwidth]{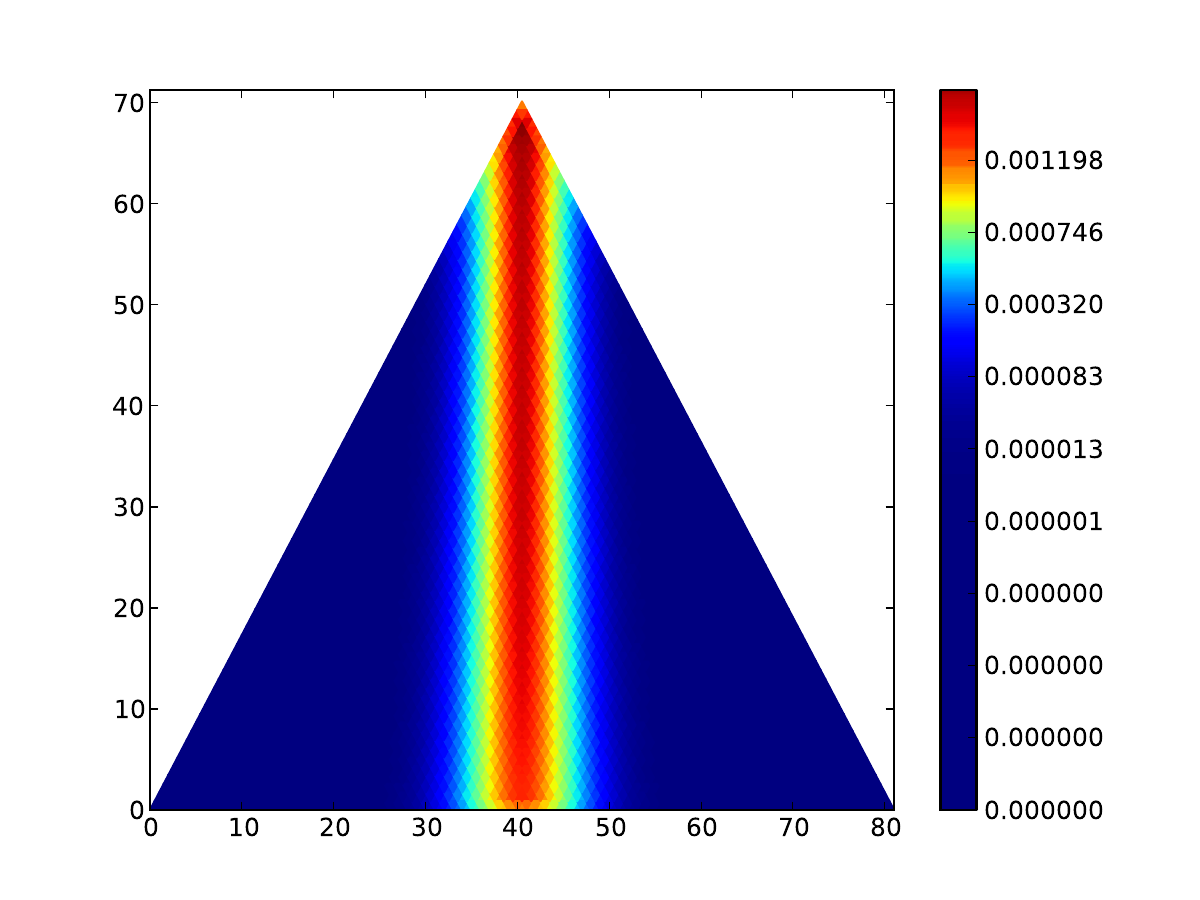}
        \end{subfigure}
        \caption{Top Row: Relative entropy of the expected next state with the current state $D_1(\bar{a})$ for Fermi incentives ($q=1$, $\beta=1$) for the Wright Fisher process with $\mu=\frac{3}{2}\frac{1}{N}$ for game matrices 2, 20, and 47 in Bomze's classification (same as Figure \ref{three_player_plots}. Middle Row: Stationary distributions for the Wright-Fisher process. Bottom Row: Stationary distributions with $\mu=\frac{1}{2}\frac{1}{N}$. The bottom row has relative entropies that are slightly different (notably the lower left vertex of the middle column plots), but they are very similar to the top row and so omitted.}
        \label{three_player_plots_wf}
\end{figure} 
\end{landscape}

\section{Extensions}

\subsection{Non-constant Populations}

Like the Moran process, the incentive and Wright-Fisher processes keep the population at a fixed size. Here we give an example that suggests that the methods used in this manuscript could be useful for processes with variably sized populations. We use the process defined in \cite{harper2013inferring}, and note that there are other models of variable population size process and dynamics in the literature, such as \cite{melbinger2010evolutionary}.

We modify the Moran process by inserting a coin-flip in each step to determine if a birth event or death event is to take place. The probability of this intermediate step can depend on the population state, and so can define a carrying capacity for the population. For instance, we could take an curve that has probability one of a birth event if the population has size $M=1$, is $1/2$ for $1 < M < N$, and is 0 for $M=N$. Clearly there are many variations on this probability, and the Moran process is a special case, if we separate the birth and death events.

Figure \ref{figure_vps} shows the stationary distribution (left plot) of such a process with the replicator incentive for the fitness landscape defined by $a=1=d$ and $b=2=c$ and the birth probabilities described in the previous paragraph. Now the ternary plot shows the population states $(a_1, a_2)$ where $a_1 + a_2 < N$; the rightward coordinate is the population size $N$. The right plot is $D_0$ of the expected next state from the current state, where the neighboring states now may have a total population size that is greater or less by one. On the right hand plot, the Lyapunov quantity is minimal when $a_1 = a_2$, as we might expect from the fitness landscape and the examples earlier in the paper. These minima occur at the maxima of the stationary distribution as expected, since for $a_1 + a_2 = M < N$, we would expect the stationary max for $M$ fixed to be $a_1=a_2=M/2$. We leave a precise extension of the main theorem to such processes as an open question.

\begin{figure}[h]
        \begin{subfigure}[b]{0.4\textwidth}
            \centering
            \includegraphics[width=\textwidth]{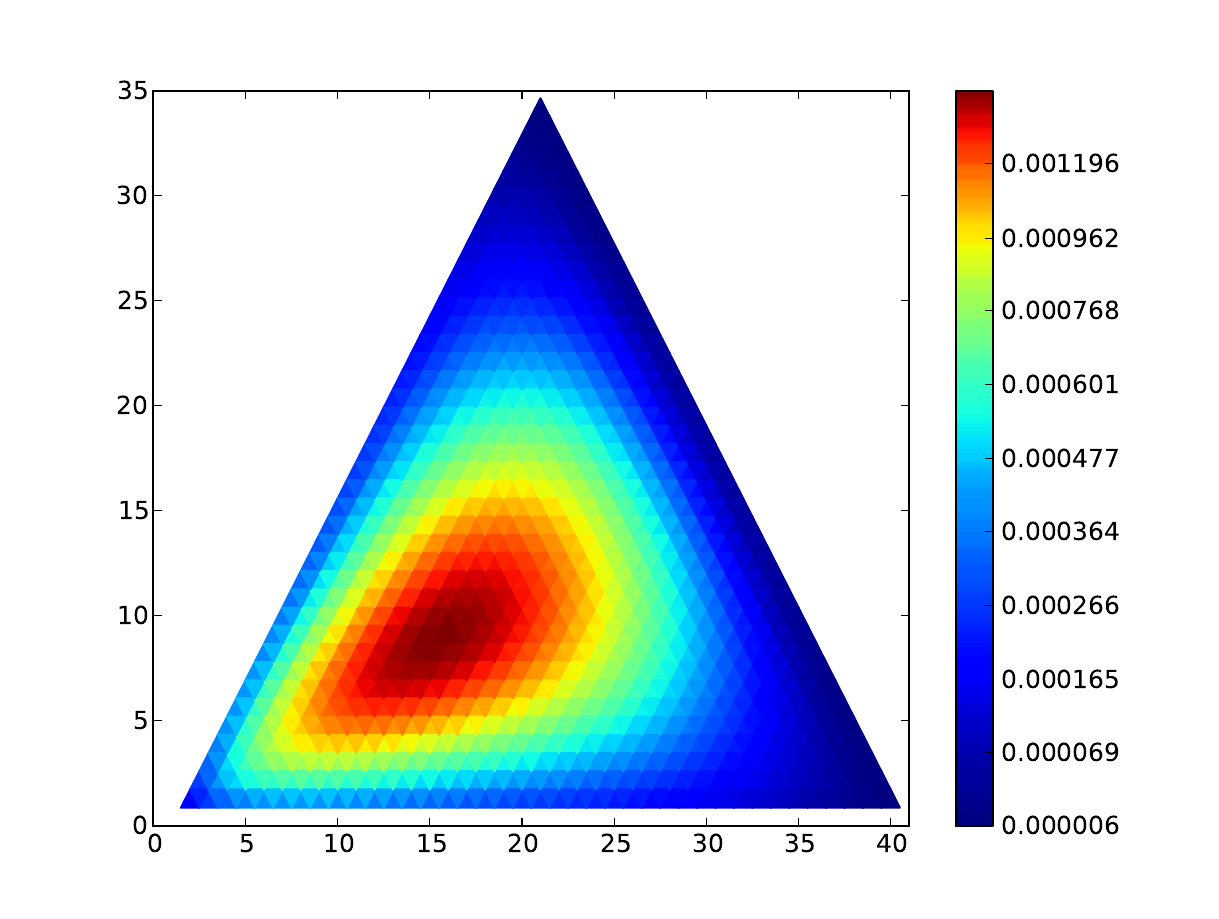}
        \end{subfigure}
        \begin{subfigure}[b]{0.4\textwidth}
            \centering
            \includegraphics[width=\textwidth]{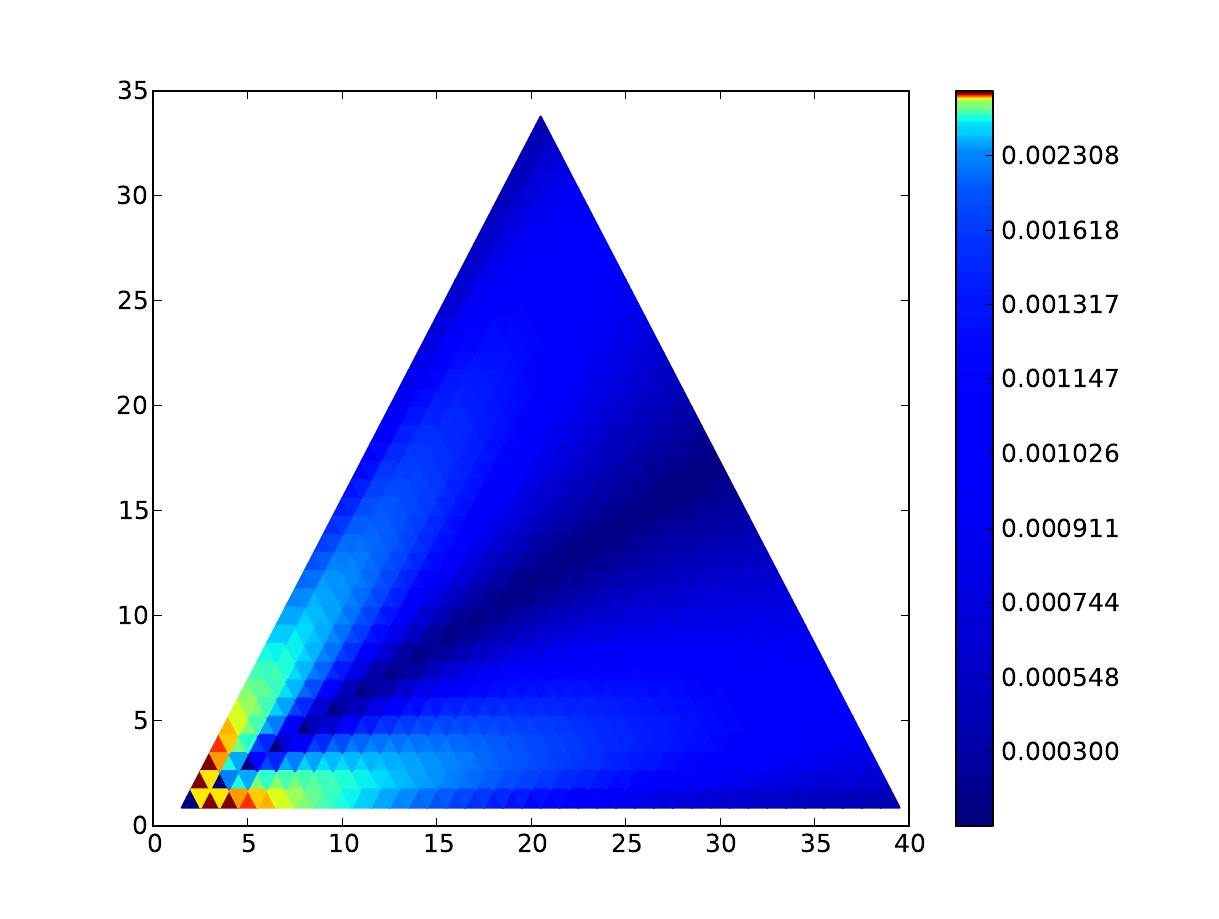}
        \end{subfigure}%
        \caption{Left Figure: Stationary distribution for a Moran-like process with non-constant population size with the replicator incentive for the fitness landscape defined by $a=1=d$ and $b=2=c$, maximum population size $N=40$, $\mu = 0.01$. The coordinates are $(a_1, a_2, a_1 + a_2)$, with the population size $a_1 + a_2$ increasing to the top and right. Right Figure: $\sqrt{D_0}$ distance between the expected next state and the current state. As expected, minima occur along the line $a_1=a_2$ with a local minima of $D_1$ at $(20,20)$. We have taken the root of $D_1$ to exaggerate the variation visually. Sigmoid (s-curve) probabilities for the birth-death decision behave similarly.}
        \label{figure_vps}
\end{figure} 

\subsection{Evolutionary Graph Theory}
Many authors now study Moran-like processes for populations distributed on graphs \cite{lieberman2005evolutionary} \cite{ohtsuki2007evolutionary}. So far we have been working in a fully-connected population, that is, a population on a complete graph in the sense that every replicator interacts with every other replicator (not to be confused with the graph defined by the transitions of the Markov processes). We outline how a result analogous to Theorem \ref{main_thm_local} could hold for populations on graphs. A birth-death process on a graph is again a Markov process, and we consider an incentive process on two-types. Let the graph be a cycle. We argue for the stationary stability of some states over others. Suppose that the replicators are distributed $A,B,A,B,\ldots$ about the cycle (i.e. a proper coloring), as in Figure \ref{figure_cycle} (left). Then any reproductive event $A \to B$ or $B \to A$ will necessarily change the state, since every vertex has adjacent neighbors of the other type. Such a state is inherently unstable, and will have a small presence in the stationary distribution. 

\begin{figure}[h]
        \begin{subfigure}[b]{0.2\textwidth}
            \centering
            \includegraphics[width=\textwidth]{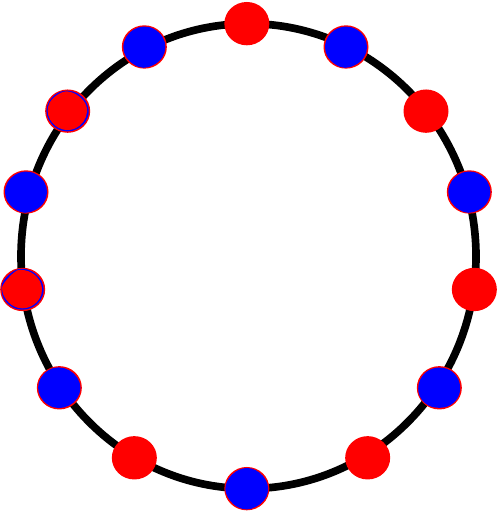}
        \end{subfigure}
        \qquad
        \begin{subfigure}[b]{0.2\textwidth}
            \centering
            \includegraphics[width=\textwidth]{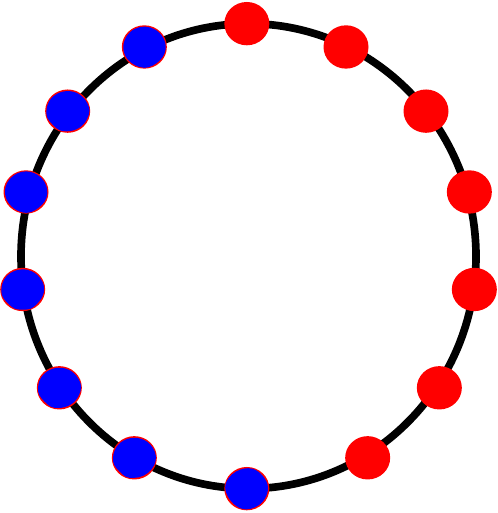}
        \end{subfigure}%
        \caption{Two configurations for an incentive process on a cycle. The left hand configuration is inherently unstable since any non-mutation replication will alter the configuration. The right-hand configuration is much more stable. For small mutation rates, only replication events on the interface between the two subpopulations can alter the configuration.}
        \label{figure_cycle}
\end{figure} 

Similarly, we can follow this reasoning to conclude that the most stable configurations will have the least possible interactions between the two types. In the case of a cycle, this means that the population is segregated into two connected subpopulations, as depicted in Figure \ref{figure_cycle} (right). All the remains is to determine the relative sizes of the two subpopulations. Proceeding in the spirit of Theorem \ref{main_thm_local}, we look for states such that the population state's probability of moving up $(i, N-i) \to (i+1, N-i-1)$ is the same as the probability of moving down $(i, N-i) \to (i-1, N-i+1)$. For very small $\mu$, only replication events at the interface of the two subpopulations will affect the graph state. Suppose that subpopulations are of distribution $a = (i, N-i)$. Then we have that the probability of increasing the size of the subpopulation of type $A_1$ (red) is simply $p_1(\bar{a}) (2/a_1) (1/2)$ since we must choose an $A_1$ individual to reproduce, it has to be one of the two on the interface between the subpopulations, and we have to randomly chose the individual of type $A_2$ (blue) to be replaced. Similarly for the subpopulation of $A_2$ individuals, and these transitions are equal when $a_2 p_1(\bar{a}) = a_1 p_2(\bar{a})$, i.e. when the subpopulations are in the ISS proportions. We have verified computationally that this is in fact the case for some two player games, such as $a=1=d$, $b=2=c$ and the neutral landscape for $2 \leq N \leq 14$. Specifically, there is a particular value of $\mu$ such that these configurations are stationary stable for each $N$ tested (when consolidated via cyclical symmetry). For $N=2$, this value is analytically computable to be $\mu = 1/3$, and $\mu \geq 1/N$ appears sufficient computationally in other cases. For larger $N$ the number of states in the Markov chain is too large for a direct approach.

A similar line of reasoning applies to the random graph, and so we conjecture that analogous results should hold for the random graph and k-regular graphs (and likely others).


\section{Discussion}

\subsection{Vanishing Mutation}

Throughout, we have explicitly assumed that the mutation rates are (at least) nonzero so as to guarantee the existence of the stationary distribution. In \cite{fudenberg2006imitation} \cite{fudenberg2006evolutionary} \cite{fudenberg2008monotone}, it is shown that as $\mu \to 0$, the stationary distribution concentrates on the boundary states for the Moran process (and some others), and is given in terms of the fixation probabilities of the process. These results complement those of this manuscript, connecting stationary stability with fixation probabilities. In other words, it is trivially the case that for $\mu = 0$, the fixation states are stationary stable for e.g. the Moran process on linear landscapes, since they are locally maximal for the stationary distribution. That they are ISS/ESS also  depends on the incentive -- this would be true for the the Moran process (e.g. $p((0,1)) = (0,1)$ in the $n=2$ case) but not necessarily for e.g. the projection incentive on the neutral landscape.

\subsection{Existence of ISS Candidates}

We have left the question of whether stationary stable states or ISS candidates exist as a prerequisite for the results presented thus far. For fixed $\mu$ and other parameters fixed, there may not be a value of $N$ such that $p(\bar{a}) = \bar{a}$ for an arbitrary incentive, but the converse is true: $\mu = \frac{n-1}{n}$, where $n$ is the number of types, gives $p_i(\bar{a}) = \frac{1}{n}$ for all $i$, so the central point of the simplex is an ISS candidate, and it is stationary stable (the process is equivalent to the neutral landscape, which has a known stationary distribution for all $n$ for this mutation matrix \cite{khare2009rates}). Computationally we have observed for fixed population size, there is often a critical value of $\mu$ that determines if there is an internal stationary maximum, which is not surprising given that the processes often fixate when $\mu \to 0$. Hence there are may be analogs to Theorems \ref{main_thm_local} \ref{main_thm_global} such that for fixed $N$, there exists a mutation matrix that yields a stationary maximum that is also evolutionarily stable, or some other relationship between $\mu$ and $N$ that achieves this.

\subsection{Large Population Limit}

We have shown above that the relative entropy $D_{KL}(E(\bar{a}), \bar{a})$ is locally minimal for evolutionary stable states, and so we can view Theorem \ref{main_thm_global} as a finite population analog of the fact that $D_{KL}(\hat{x}||x)$ is a Lyapunov function for the replicator dynamic when $\hat{x}$ is evolutionarily stable. Explicitly, Traulsen et al show \cite{traulsen2005coevolutionary} \cite{traulsen2006coevolutionary} \cite{traulsen2006stochasticity} \cite{pacheco2006active} \cite{zhou2011fixation} essentially that the Moran process maps to the stochastic replicator equation with mean dynamic $\dot{x} = T^{+}(x) - T^{-}(x)$ (in the $n=2$ case). The stable rest points of this dynamic are those that satisfy the ISS candidate condition when we replace $\bar{a}$ with $x$, and we have that the relative entropy is locally minimal at these points. Moreover, this limiting process shows that any stable rest point of the replicator dynamic must be a limit of stationary stable states for some sufficiently large $N$ (the stationary states are by definition rational and the rest point could be irrational).

What happens to the stationary distribution in the large population limit? Consider the explicit stationary distribution in Equation \ref{stationary_example}. For large $N$, we have that the stationary distribution converges to a binomial distribution with $p = 1/2$. Scaling to the simplex $[0,1]$, we have that the stationary distribution is, by the normal approximation to the binomial, a normal distribution centered at $1/2$ with variance $1/(4N)$. As $N \to \infty$, the stationary distribution becomes a delta function centered on the ESS. Furthermore, since the replicator equation is deterministic, the expected next state is also a deterministic quantity, and naturally replaced by the ESS $\hat{x}$, and so the Lyapunov quantity becomes the function $D_{KL}(\hat{x}||x)$. This is essentially why the stationary distribution and the relative entropy capture the properties of the vector field of the replicator equation, as in Figure \ref{figure_graphical_abstract}. We highly encourage the reader to compare Figure \ref{figure_rsp_lyapunov} to Figure 1.4 of \cite{traulsen2009stochastic} for a striking visualization of the relationship between a Lyapunov quantity (essentially equivalent to the relative entropy) for the replicator dynamic and the functions $D_d$ that we have defined in this work. Sandholm et al also derive a number of large population limits of the stationary distribution for similar processes \cite{sandholm2007simple} and \cite{sandholm2014large}.

The key difference between the distance functions is the natural geometry of the simplex, the Shahshahani geometry \cite{shahshahani1979new}, is closely associated with the relative entropy. Forward-invariance of the replicator equation is a property acquired because of the limit $N \to \infty$, and not true for the Markov processes considered above. This substantially changes the character of the limits as $x_i \to 0$ since for integral population states, the process gets no closer to zero than $1/N$.

\subsection{Conclusion}

We have introduced the concept of stationary stability for finite populations and have shown that it captures the traditional notion of evolutionary stability. In the large population limit, our results limit to the known Lyapunov theory of evolutionary stability via Traulsen et al's mapping of the Moran process to the replicator equation. We have demonstrated that stationary stability works for all mutation rates and for a large class of incentives, connects with the fixation theory of the Moran process, and likely extends to finite populations on graphs.

Crucially, we have given Lyapunov quantities that are very easy to compute as a means of understanding the stationary distribution, which can be very difficult to study analytically, and for which the stationary distribution can not be characterized in closed form (currently). These quantities derive from information-theoretic distance functions ranging from the Euclidean distance to the relative entropy (or Kullback-Leibler divergence).

\subsection*{Methods}
All computations were performed with python code available at \url{https://github.com/marcharper/stationary}. This package can compute exact stationary distributions for detail-balanced processes and approximate solutions for all other cases mentioned in this manuscript. All plots created with \emph{matplotlib} \cite{Hunter:2007}, except the vector field in the graphical abstract, created with Dynamo \cite{franchetti2012introduction}.

\subsection*{Acknowledgments}

This work was partially supported by the Office of Science (BER), U. S. Department of Energy, Cooperative Agreement No. DE-FC02-02ER63421.

\bibliography{ref}
\bibliographystyle{plain}

\end{document}